\documentclass[a4paper,11pt,leqno]{article}
\usepackage[utf8]{inputenc}

\usepackage[english]{babel}
\usepackage{amsmath,amsfonts,amssymb}
\usepackage{amsthm,ifthen}
\usepackage[mathscr]{eucal}
\renewcommand{\thesection}{\S~\arabic{section}.}
\newcounter{subsec}[section]
\newcommand{\subSect}{\refstepcounter{subsec}%
\arabic{section}.\thesubsec.}
\newtheorem{Teo}{Theorem}

\newtheorem{Remark}{Remark}[section]
\newtheorem{Pro}{Proposition}[section]
\newtheorem{lemma}[Remark]{Lemma}

\renewcommand{\theRemark}{\arabic{section}.\arabic{Remark}}
\renewcommand{\thePro}{\arabic{section}.\arabic{Pro}}

\newtheorem{Teoo}[Remark]{Theorem}
\newtheorem{Deff}[Remark]{Definition}
\newcommand{\var}{\ensuremath{\mathop\mathrm{var}}}
\numberwithin{equation}{section}
\newtheorem{Sti}[Remark]{Stipulation}

\newcommand{\N}[1][N]{\ensuremath{\mathbb{#1}}}
\renewcommand{\cite}[2][]{\ifthenelse{\equal{#1}{}}{[#2]}%
{[#2, #1]}}
\renewcommand{\P}[1][P]{\ensuremath{\mathscr{#1}}}
\newcommand{\K}{\P[K]}
\newcommand{\card}{\mathop{\mathrm{card}}}

\newcommand{\B}{\P[B]}
\newcommand{\Hh}{\P[H]}
\newcommand{\C}{\P[C]}
\newcommand{\F}{\P[F]}

\renewcommand{\theequation}{\arabic{section}.\arabic{equation}}
\begin{document}\large\sloppy
\begin{center}
\Large
Nonlinear Piecewise
Polynomial Approximation and Multivariate $BV$ spaces of a Wiener--L.~Young 
Type. I.\\
Yu. Brudnyi (Technion, Haifa)
\end{center}

\begin{abstract}
The named space denoted by $V_{pq}^k$ consists of $L_q$ functions on $[0,1)^d$ 
of bounded $p$-variation of order $k\in\N$. It generalizes the classical spaces 
$V_p(0,1)$ ($=V_{p\infty}^1$) and $BV([0,1)^d)$ ($V_{1q}^1$ where $q:=\frac 
d{d-1}$) and closely relates to several important smoothness spaces, e.g., to 
Sobolev spaces over $L_p$, $BV$ and $BMO$ and to Besov spaces.

The main approximation result concerns the space $V_{pq}^k$ of 
\textit{smoothness} $s:=d\left(\frac1p-\frac1q\right)\in(0,k]$. It asserts the 
following:

Let $f\in V_{pq}^k$ are of smoothness $s\in(0,k]$ and $N\in\N$. There exist a 
family $\Delta_N$ of $N$ dyadic subcubes of $[0,1)^d$ and a piecewise 
polynomial $g_N$ over $\Delta_N$ of degree $k-1$ such that
\[
 \|f-g_N\|_q\leqslant CN^{-s/d}|f|_{V_{pq}^k}.
\]
This implies the similar results for the above mentioned smoothness spaces, 
in 
particular, solves the going back to the 1967 Birman--Solomyak paper \cite{BS} 
problem of approximation of functions from $W_p^k([0,1)^d)$ in $L_q([0,1)^d)$ 
when ever $\frac kd=\frac1p-\frac1q$ and $q<\infty$.

Besov spaces.
\end{abstract}

\section{Introduction}

\subSect{} The present paper is initiated by two results removed from each 
other for more than thirty years. The first, the 1967 pioneering 
paper 
\cite{BS}, asserts the following:
\begin{Teo}\label{TeoA}
Given $f\in W_p^k([0,1)^d)$, $N\in\N$ and $q<\infty$ satisfying 
\[
 \frac kd>\frac1p-\frac1q
\]
there exist a partition $\Delta_N$ of $[0,1)^d$ into at most $N$ dyadic 
subcubes and a piecewise polynomial $g_N$ on $\Delta_N$ of degree $n-1$ such 
that 
\begin{equation}\label{1.1}
\|f-g_N\|_q\leqslant CN^{-k/d} \sup\limits_{|\alpha|=k}
 \|D^\alpha f\|_p;
\end{equation}
the constant $C>0$ is independent of $f$ and $N$ and $C\to\infty$ as $q$ 
tends to the \textit{Sobolev limiting exponent} $q^*:=\left(\frac1p-\frac 
kd\right)^{-1}$.
\end{Teo}

Using the compactness argument from \cite[\S~5]{BS} one can prove that 
validity of \eqref{1.1} for $q=q^*$ implies (incorrect) compactness of
embedding 
$W_p^k\subset L_{q^*}$. 

This leads to the following:

\textbf{Problem}. Find a generalization of Theorem~\ref{TeoA} for $q=q^*$. 

For the special case $k=p=1$, $d=2$ the answer was given in the 1999 paper 
\cite{CDPX} by A.~Cohen, De Vore, Petrushev and Hong Xu; the case $d>2$ was 
than proved by Wojtashchyk \cite{W}. The result states:
\begin{Teo}\label{TeoB}
Given $f\in W_1^1([0,1)^d)$, $d\geqslant 2$, and $N\in\N$ there exist a 
partition $\Delta_N$ of $[0,1)^d$ into at most $N$ $d$-rings (differences of 
two dyadic subcubes) and a piecewise constant function $g_N$ on $\Delta_N$ such 
that
\begin{equation}\label{1.2}
 \|f-g_N\|_{q^*} \leqslant c(d) \sup\limits_{\|\alpha\|=1} \|D^\alpha f\|_1.
\end{equation}
\end{Teo}

Hereafter $c(x,y,\ldots)$ denotes a positive constant \textit{depending only} 
on the parameters in the brackets.
\begin{Remark}\label{R1.1}\em
(a) Theorem~\ref{TeoB} is proved for $L_1([0,1)^d)$ functions whose 
distributions derivatives of the first order are bounded Radon measures, see 
the cited papers. However, this more general result follows directly from 
\eqref{1.2}, see sec.~4.7 below.

(b) The partition $\Delta_N$ in Theorem \ref{TeoB} can be replaced by a 
covering of $[0,1)^d$ by at most $N$ dyadic cubes. This result is equivalent to 
the previous, see Remark~4.9 
below.
\end{Remark}

\subSect{} In this paper, we prove a general result on $N$-term piecewise 
polynomial approximation implying as consequences the similar results for a 
good few spaces of smoothness functions.

To illustrate the main result we formulate its consequence giving the solution 
of the Birman--Solomyak problem and the generalization of Theorem~\ref{TeoB} to 
functions of higher smoothness.
\begin{Teoo}\label{Teoo-1.2}
(a) Given $f\in W_p^k([0,1)^d)$, $n\in\N$ and $q^*<\infty$ such that
\[
 \frac kd=\frac1p-\frac1{q^*},\quad
 d\geqslant 2,
\]
there exist a covering $\Delta_N$ of $[0,1)^d$ by at most $N$ dyadic subcubes 
and a family of polynomials $\{P_Q\}_{Q\in \Delta_N}\subset \P_{k-1}$ 
(of degree $k-1$) such that\footnote{hereafter $f_S$ stands for the 
\textit{indicator} (characteristic) function of a set $S$}
\begin{equation}\label{1.3}
\left\|f-\sum\limits_{Q\in\Delta_N} P_Q\cdot 1_Q\right\|_{q^*}
\leqslant c(k,d)
N^{-k/d}\sup\limits_{|\alpha|=k} \|D^\alpha f\|_p.
\end{equation}
(b) For $p:=1$, hence, $q^*=\frac d{d-k}$, the previous holds for $f\in 
L_1$, whose derivatives of order $k$ are bounded Radon 
measures.
\end{Teoo}

The associated seminorm of the latter function space denoted by 
$BV^k([0,1)^d)$ is given by
\begin{equation}\label{1.4}
|f|_{BV^k}:=\sup\limits_{|\alpha|=k} \var \limits_{[0,1)^d} (D^\alpha 
f).
\end{equation}
\begin{Remark}\label{Remark1.3}\em
The result can be reformulated equivalently with the covering 
$\Delta_N$ substituted for a \textit{partition} of $[0,1)^d$ into at most 
$N$ $d$-rings. 
\end{Remark}

\subSect{} The formulated aim will be achieved by using the $BV$ 
spaces of integrable on $[0,1)^d$ functions of arbitrary smoothness introduced 
in \cite[sec.~4.5]{B-71}. To motivate definition of the corresponding space 
denoted by $V_{pq}^k$ we begin with a model case, the Wiener--L.~Young 
space $V_p$, 
whose associated seminorm we present in the following equivalent form:
\begin{equation}\label{1.6}
 \var{}_p f:=\sup\limits_\Delta\left(\sum\limits_{I\in\Delta} osc
(f;I)^p\right)^{1/p}
\end{equation}
where $\Delta$ runs over disjoint families of intervals $I=[a,b)\subset [0,1)$ 
and
\begin{equation}\label{1.7}
 osc (f;I):=\sup\limits_{x,y\in I} |f(x)-f(y)|.
\end{equation}

To obtain the required seminorm of $V_{pq}^k$ denoted by $\var_p^k(\cdot; L_q)$ 
we replace in \eqref{1.6} intervals by cubes $Q\subset [0,1)^d$, in 
\eqref{1.7} the first difference by the $k$-th one, and the uniform norm by 
$L_q$ norm. This gives the following:
\begin{Deff}\label{Deff1.4}
The seminorm $f\mapsto\var_p^k(f;L_q)$ is a function on $L_q([0,1)^d)$ given by
\begin{equation}\label{1.8}
 \var{}_p^k (f;L_q):= \sup\limits_{\Delta} \left\{ \sum\limits_{Q\in\Delta}
 osc_q^k (f;Q)^p\right\}^{1/p}
\end{equation}
where $\Delta$ runs over disjoint families of cubes $Q\subset[0,1)^d$ and
\begin{equation}\label{1.9}
osc_q^k (f;Q):= \sup\limits_{h\in \N[R]^d}
 \|\Delta_h^k\|_{L_q(Q_{kh})}; 
\end{equation}
here $\Delta_h^k:=\sum\limits_{j=0}^k (-1)^{n-j} 
\begin{pmatrix}
 k\\j
\end{pmatrix}
\delta_{jh}$ and
$Q_{kh}:=\{x\in \N[R]; x+jh\in Q,j=0,1,\ldots,k\}$.
\end{Deff}

The important characteristic of the  space $V_{pq}^k$ is its 
\textit{smoothness} introduced by the following:
\begin{Deff}\label{Deff1.5}
Smoothness of the space $V_{pq}^k$ denoted by $s(V_{pq}^k)$ is a real number 
given by 
\begin{equation}\label{1.10}
 s(V_{pq}^k):=d\left(\frac1p-\frac1q\right).
\end{equation}
\end{Deff}

This concept closely relates to differential and approximation 
properties of $V_{pq}^k$ functions. In fact, a function with $s(V_{pq}^k)=s$ 
belongs to the Taylor class $T_q^s(x)$ a.e. if $0<s<k$ and $t_q^k(x)$ a.e. 
if $s=k$, see \cite[\S~2]{B-94}.

Moreover, its order of $N$-term approximation in $L_q([0,1)^d)$ by piecewise 
polynomial is, as we will see, $N^{-s/d}$ for $0<s\leqslant k$.

In particular, proof of Theorem~\ref{Teoo-1.2} is based on the equality
\[
 V_{pq^*}^k=W_p^k
\]
and the fact that 
\[
 s(V_{pq^*}^k):=d\left(\frac1p-\frac1{q^*}\right)=k
\]
that allow to derive it directly from the corresponding result for 
$V_{pq^*}^k$.

\subSect{} The outline of the paper is the following:

The main result, Theorem~2.1, is formulated in \S~2 along with its consequences 
describing the similar approximation results for classical smoothness spaces 
(one of them is Theorem~\ref{Teoo-1.2} formulated above)

In \S~3, we prove properties of $V_{pq}^k$ spaces essential for the proof 
of Theorem~2.1. The first asserts that a function $f\in V_{pq}^k$ can be weakly 
approximated  in $L_q$ by $C^\infty$ functions whose $(k,p)$-variations 
bounded by $\var_p^k(f;L_q)$. The special case for the space $BV([0,1)^d)$ 
($=V_{11}^1$), see, e.g., in \cite[Thm.~5.3.3]{Z}.

The second results estimates polynomial approximation of order $k-1$ for $f\in 
V_{pq}^k(Q'\setminus Q'')$ via its $(k,p)$-variation; here $Q''\subsetneqq Q'$ 
are dyadic cubes.

The latter result essentially uses a covering lemma proved in collaboration 
with V.~Dolnikov; its proof is presented in Appendix~I.

For $k=1$ (approximation by constants) the result was proved in \cite{CDPX} for 
$d=2$ and in \cite{W} for $d>2$ by two different methods that do not work 
for 
$k>1$.

The main result, Theorem 2.1 is proved in \S~4 and its consequences for the 
classical smoothness spaces in \S~5 (one of them, Theorem~\ref{Teoo-1.2}, had 
been yet formulated). The approximation algorithm used in construction of the 
family of dyadic cubes for Theorem~\ref{Teo-2.1} is presented in Appendix~II.

Its primary version was developed to prove the similar to 
Theorem~\ref{Teoo-1.2} result for $N$-term approximation of functions from 
Besov spaces by $B$--splines; the result announced in \cite{BI-87} and proved in 
\cite{BI-92}.

The special case of Theorem~2.1 for functions with absolutely continuous 
$(k,p)$-variation proved many times ago and announced in \cite{B-2002}. This 
result allows to prove all consequences of Theorem~2.1 presented in \S~2 but 
only much 
later the author understood how to derive Theorem~2.1 from this special case, 
see Section~4.7.\smallskip

\noindent {\em Acknowledgement}. I am very grateful to my friends Irina 
Asekritova, Vladimir Dolnikov and Natan Kruglyak for their lively interest to 
and useful discussions of this paper.

\section{Formulations of the Main Results}

Throughout the paper we use the following notions and notations.

A \textit{cube} denoted by $Q$, $Q'$, $K$ etc. is a set of $\N[R]^d$ homothetic 
to the (half--open) \textit{unit cube}
\begin{equation}\label{2.1}
  Q^d:=[0,1)^d.
\end{equation}
$\P[D](Q)$ denotes the family of \textit{dyadic cubes} of $Q$, i.e., cubes of 
the form
\begin{equation}\label{2.2}
 K:=2^{-j}(Q+\alpha)
\end{equation}
where $j\in\N[Z]_+:=\{0,1,2,\ldots,\}$ and $\alpha\in\N[Z]^d$.

Further, $\P_l=\P_l(\N[R]^d)$ is \textit{the space of polynomials in} 
$x:=(x_1,x_2,\ldots,x_d)$ \textit{of degree} $l$ while $\P_l(\Delta)$ denotes 
the \textit{space of piecewise polynomials on a set $\Delta \subset D(Q)$ of 
degree $l$.} 

In other words, 
\begin{equation}\label{2.3}
 \P_l(\Delta):=\{f\in L_\infty(Q);
 f=\sum\limits_{K\in\Delta} P_K\cdot 1_K\}
\end{equation}
where $\{P_K\}_{K\in\Delta}\subset \P_l$.
\begin{Teoo}\label{Teo-2.1}
(a) Let $f\in V_{pq}^k(Q^d)$ where the smoothness $s:=s(V_{pq}^k)$, 
see \eqref{1.10}, and 
$d$, $p,q$ be such that 
\begin{equation}\label{2.4}
d\geqslant2,\,\, 0<s\leqslant k\text{ and } 1\leqslant p<q<\infty. 
\end{equation}
Given $N\in\N$ there exist a covering $\Delta_N\subset \P[D](Q^d)$ of $Q^d$ 
with 
card $\Delta_N\leqslant N$ and $g_N\in \P_{k-1}(\Delta_N)$ such that 
\begin{equation}\label{2.5}
 \|f-g_N\|_q\leqslant c(d)N^{-s/d} |f|_{V_{pq}^k}.
\end{equation}
The same is true for $q=\infty$, i.e., for $s/d=1/p$, if $f$ is uniformly 
continuous on $Q^d$.

(b) The covering $\Delta_N$ can be replaced by a 
partition of $Q^d$ into at most $N$ dyadic $d$-rings.
\end{Teoo}  
\begin{Sti}\label{Sti-2.2}\em
We delete the symbol $Q^d$ from the next notations writing, e.g., $\P[D]$, 
$V_{pq}^k$, $L_q$ instead of $\P[D](Q^d)$, $V_{pq}^k(Q^d)$, $L_q(Q^d)$ if it 
does 
not lead to misunderstanding.
\end{Sti}

The first consequence of the main result, Theorem~\ref{Teoo-1.2}, immediately 
follows from Theorem~\ref{Teo-2.1}(a) and the inequality
\begin{equation}\label{2.6}
 \var\nolimits_p^k(f;L_{q^*})\leqslant c
 \left\{ 
 \begin{array}{cll}
  |f|_{W_p^k} & \text{if} & p>1\\
  |f|_{BV^k} & \text{if} & p=1,
 \end{array}
 \right.
\end{equation}
here $c=c(k,d,q^*)$ and $q^*:=\left(\frac kd-\frac1p\right)^{-1}$.

This and analogous embedding results for Besov spaces are presented in \S~5.

Let now $\dot B_p^{\lambda \theta}:=\dot B_p^{\lambda \theta}(Q^d)$ be the 
homogeneous Besov space defined by the seminorm
\begin{equation}\label{2.7}
 |f|_{B_p^{\lambda\theta}}:=
 \left\{
 \int\limits_0^1\left(\frac{\omega_k(t;f;L_p)}{t^\lambda}
 \right)^\theta\frac{dt}t\right\}^{1/\theta} 
\end{equation}
where $k=k(\lambda):=\min\{n\in\N;n>\lambda\}$ and $\omega_k(\cdot; f; L_p)$ is 
the $k$-th modulus of continuity of $f\in L_p$, see e.g., \cite{N} or 
\cite{DL} for its 
definition.

The first result concerns the ``diagonal'' Besov space $\dot B_p^\lambda:= \dot 
B_p^{\lambda p}$, $1\leqslant p<\infty$.
\begin{Teoo}\label{Teo-2.3}
Let $f\in \dot B_p^\lambda$ and $d,p,q,\lambda$ be such that 
\[
 d\geqslant2 , 1\leqslant p<q<\infty \,\,\text{ and }\,\,\frac\lambda 
d=\frac1p-\frac1q.
\]

Given $N\in\N$ 
there exist a covering $\Delta_N\subset \P[D]$ of $Q^d$ by at most $N$  cubes 
and 
$g_N\in\P_{k-1}(\Delta_N)$ such that
\[
 \|f-g_N\|_q\leqslant c(k,d) N^{-\lambda/d} |f|_{B_p^\lambda}.
\]
\end{Teoo}

The second result concerns approximation in the uniform norm ($q=\infty$).
\begin{Teoo}\label{Teo-2.4}
 Let $f\in \dot B_p^{\lambda1}$ and 
\[
 d\geqslant2, \quad
 1\leqslant p<\infty,
 \quad
 \frac\lambda d=\frac1p.
\]

Given $N\in\N$ there exist $\Delta_N\subset \P[D]$ satisfying the condition of 
Theorem~\ref{Teo-2.3} and $g_N\in\P_{k-1}(\Delta_N)$ such that 
\[
 \|f-g_N\|_\infty\leqslant c(\lambda,d) N^{-\lambda/d} |f|_{B_p^{\lambda1}}.
\]
\end{Teoo}
\begin{Remark}\label{Remark-2.5}\em
(a) The seminorms $|f|_{V_{pq}^k}$ and $|f|_{B_p^{\lambda\theta}}$ are 
naturally extended by the formulas \em(3.1$^c$)\em{} and \eqref{2.7} to $p<1$. 
Moreover. unlike \eqref{2.6}, the corresponding embedding results are true for 
this case. For instance,
\[
 |f|_{B_p^\lambda}\leqslant c(\lambda,d,p)  |f|_{V_{p\infty}^k}
\]
for $0<p\leqslant 1$ and $\frac\lambda d=\frac1p$.

Therefore Theorems~\ref{Teo-2.3} and \ref{Teo-2.4} are true for $p<1$ as well.

(b) For $B$--splines and, more generally, for refinable functions with the 
Strang--Fix condition the analog of Theorem~\ref{Teo-2.3} was proved in 
\cite{BI} and \cite{DJP}, respectively. The latter result was extended to $p<1$ 
in \cite{DPY}.

Formally, Theorem~\ref{Teo-2.3} does not follow from these results but it can 
be surely proved by the approximation algorithms used in the cited 
papers.
\end{Remark}

\section{Spaces $V_{pq}^k$}

\subSect{} \textit{$(k,p)$-variation}. The named object is a set-function 
defined by \eqref{1.8} with $Q^d$  substituted for a measurable set $S\subset 
\N[R]^d$ of nonempty interior.

We, however, prefer to use an equivalent definition more suitable for our aim. 
There we replace $osc_p^k$ by \textit{local best approximation}, a set-function 
given for $f\in L_q^{loc}(\N[R]^d)$ and $S\subset \N[R]^d$ by
\begin{equation}\label{3.1}
  E_{k}(f;S;L_q):=
 \inf\limits_{m\in\P_{k-1}} \|f-m\|_{L_q(S)}.
\end{equation}

This replacement gives the following:
\begin{Deff}\label{Deff-3.1}
$(k,p)$-variation of a function $f\in L_q^{loc}(\N[R]^d)$ is a set-function on 
subsets $S\subset \N[R]^d$ with nonempty interior given by 
\[
 \var{}_{{p}}^{{k}} (f;S;L_q):=
 \sup\limits_{\Delta} \left\{
 \sum\limits_{Q\in\Delta} E_{{k}} (f;Q;L_q)^p \right\}^{1/p}
\leqno(\ref{3.1}^a)
\]
where $\Delta$ runs over all \textit{disjoint} families of cubes $Q\subset S$.
\end{Deff}

Equivalence of this definition to the previous follows from the main result of 
\cite{B-70} implying, e.g., the next two-sided inequality with constants 
depending only on $k$:
\[
 osc_p^k(f;Q;L_q)\approx E_k(f;Q;L_q).
\leqno(\ref{3.1}^b)
\]

It should be emphasized that in the forthcoming text all definitions and 
results involving the space $V_{pq}^k$ use Definition \ref{Deff-3.1}. In 
particular, the associated seminorm of this space is
\[
 \|f\|_{V_{pq}^k}:=
\sup\limits_{\Delta} \left\{ \sum\limits_{Q\in\Delta} E_k(f;Q;L_q)^p 
 \right\}^{1/p}
\leqno(\ref{3.1}^c)
\]
where $\Delta$ runs over all disjoint families of $Q\subset Q^d$.

Now we present some basic properties of $(k,p)$-variation beginning with the
directly following from Definition~\ref{Deff-3.1}$^a$.

In its formulation, $A(\P[D])$ denoted the $\sigma$-\textit{algebra generated 
by 
dyadic cubes}.
\begin{Pro}[\em Subadditivity]\label{Pro-3.1}
Let $\{S_i\}$ be a disjoint families of sets from $A(\P[D])$. Then 
\begin{equation}\label{3.1.a}
 \left\{\sum\limits_i\var{}_p^k(f;S_i:L_q)^p\right\}^{1/p}
 \leqslant \var{}_p^k (f;\bigcup\limits_iS_i;L_q).
\end{equation}
(Lower semicontinuity) If $\{f_j\}$ converges in $L_q$ to a function $f$, then
\begin{equation}\label{3.2}
  \var{}_p^k(f;S;L_q)\leqslant
  \mathop{\underline{\lim}}\limits_{j\to\infty} \var{}_p^k (f_j;S;L_q).
\end{equation}
\end{Pro}
\begin{proof}
Let $\Delta:=\{Q\}$ be a disjoint family of cubes and 
\begin{equation}\label{3.3}
 \var{}_p^k(f;\Delta;L_q):=\left\{\sum\limits_{Q\in\Delta}
 E_k(f;Q;L_q)^p\right\}^{1/p}.
\end{equation}

If $\{\Delta_i\}$ is disjoint and $\bigcup\limits_{Q\in\Delta_i}Q\subset S_i$ 
then 
\[
 \sum\limits_i\var{}_p^k (f;\Delta_i;L_q)^p=
 \var{}_p^k (f;\bigcup\Delta_i;L_q)
 \leqslant \var{}_p^k (f; \bigcup S_i;L_q) 
\]
and it remains to take supremum over each $\Delta_i$ to prove \eqref{3.1.a}.

The property \eqref{3.2} is proved similarly.
\end{proof}

A more substantive property of $(k,p)$-variation gives the next result.
\begin{Pro}\label{Pro-3.2}
Let a $C^\infty$ function $f$ belong to the space $V_{pq}^k$ of smoothness 
${s}\leqslant {k}$. Then uniformly in $S\in A(\P[D])$
\begin{equation}\label{3.4}
 \lim\limits_{|{S}|\to0}
 \var{}_{p}^{k} (f;S;L_q)=0.
\end{equation}
\end{Pro}
\begin{proof}
Let $\Delta$ be a disjoint family of cubes $Q\subset S$. By the Taylor formula
\[
 E_k(f;Q;L_q)\leqslant
 |Q|^{1/q} E_k(f;Q;{C})\leqslant
 c(k,d)|Q|^{1/q+k/d} \max\limits_{|\alpha|=k}
 \max\limits_Q |D^\alpha f|;
\]
this implies
\[
 \var{}_p^k(f;\Delta;L_q) \leqslant
 c(k,d) \left\{
 \sum\limits_{Q\in\Delta} |Q|^{p\left(\frac1q+\frac kd\right)}
 \right\}^{1/p}
 |f|_{C^k(Q^d)}.
\]
Since $p\left(\frac1q+\frac kd\right)\geqslant p\left(\frac1q+\frac 
sd\right)=1$, the sum  here is bounded by $\left\{\sum\limits_{Q\in \Delta} 
|Q|\right\}^{1/p}\leqslant |S|^{1/p}$, and therefore
\[
  \var{}_p^k(f;S;L_q):=
 \sup\limits_\Delta \var{}_p^k (f;\Delta;L_q)
\to0
 \text{ as } |S|\to0.\quad
 \qedsymbol
\]\renewcommand{\qed}{}
\end{proof}

\subSect{} \textit{ $C^\infty$ Approximation of \ $V_{pq}^k$ Functions.}
Since the space $V_{pq}^k$ is, in general, nonseparable, $C^\infty$ 
approximated functions form a proper subspace of $V_{pq}^k$. However, a weaker 
form of $C^\infty$ approximation is true.
\begin{Teoo}\label{Teo-3.3}
Let a function $f$ belong to $V_{pq}^k$ if $q<\infty$ and to $V_{p\infty}^k 
\bigcap C(\N[R]^d)$, otherwise. Assume that $Q$ is a subcube of $Q^d$ 
such that
\begin{equation}\label{3.5}
 dist (Q;\N[R]\setminus Q^d)>0.
\end{equation}
Then there exists a sequence $\{f_j\}\subset C^\infty(\N[R]^d)$ such that 
\begin{equation}\label{3.6}
 \lim\limits_{j\to\infty} f_j =f
 \text{ (convergence in $L_q(Q)$)}
\end{equation}
and, moreover,
\begin{equation}\label{3.7}
 \sup\limits_{j} \var{}_p^k (f_j;Q;L_q)
 \leqslant \var{}_p^k (f;L_q).
\end{equation}
\end{Teoo}
\begin{proof}
Let $f_\varepsilon$ be a \textit{regularizer} of $f$ given by
\begin{equation}\label{3.8}
f_\varepsilon(x):= \int\limits_{Q_\varepsilon}
 f(x-\varepsilon y)\varphi(y)\,dy,\quad
 x\in Q,
\end{equation}
where $\varphi\in C^\infty(\N[R]^d)$ is a test function, i.e., 
\begin{equation}\label{3.9}
 \varphi\geqslant0,
 \int\varphi\,dx=1
 \text{ and } supp\varphi\subseteq[-1,1]^d,
\end{equation}
where $\varepsilon>0$ is such that
\begin{equation}\label{3.10}
 Q_\varepsilon:=Q+[-\varepsilon,\varepsilon]^d\subseteq Q^d,
\end{equation}
see \eqref{3.5}.

Now \eqref{3.9} and the Minkowski inequality yield
\[
 \|f-f_\varepsilon\|_{L_q(Q)} 
 \leqslant \sup\limits_{|y|\leqslant1} \|f(\cdot-\varepsilon y)-f\|_{L_q( Q)}.
\]
Since the right-hand side tends to $0$ as $\varepsilon\to0$ for $q<\infty$ and 
for $q=\infty$ if $f\in C(\N[R]^d)$, \eqref{3.6} follows.
\end{proof}

To proceed we need the following:
\begin{lemma}\label{lem-3.4}
It is true that
\begin{equation}\label{3.11}
 E_k(f_\varepsilon;Q;L_q)
 \leqslant \int\limits_{|y|\leqslant1}
 E_k(f;Q-\varepsilon y;L_q)\varphi(y)\,dy.
\end{equation}
\end{lemma}
\begin{proof}
It suffices to prove \eqref{3.11} for $q<\infty$ and then pass $q$ to $+\infty$.

Let $q<\infty$ and $q'$ denote the conjugate exponents. By $\P_{k-1}^\perp(Q)$ 
we denote a set of functions $g\in L_{q'}(\N[R]^d)$ such that 
\begin{equation}\label{3.12}
 \|g\|_{L_{q'}}=1,
 \,\,
 supp\, g\subset Q,
 \,\,
 \int x^\alpha g(x)\,dx=0,
 \,\,
 |\alpha|\leqslant k-1.
\end{equation}

By the duality of $L_q$ and $L_{q'}$ 
\begin{equation}\label{3.13}
  E_k(f_\varepsilon;Q;L_q)=
 \sup\left\{ \int\limits_Q f_\varepsilon g\,dx;
 g\in\P_{k-1}^\perp(Q)\right\}.
\end{equation}

On the other hand,
\begin{equation}\label{3.14}
  \left|\int\limits_Q f_\varepsilon g\,dx\right|\leqslant
 \int\limits_{\N[R]^d} \left|\,\,
 \int\limits_{Q-\varepsilon y} f(x) g(x+\varepsilon y)\,dx\right|
 \varphi(y)\,dy
\end{equation}
and the function $x\mapsto g(x+\varepsilon y)$, $x\in Q$, clearly, belongs to 
the set $\P_{k-1}^\perp (Q-\varepsilon y)$. Therefore for every polynomial $m$ 
of degree $k-1$
\[
 \int\limits_{Q-\varepsilon y} f_\varepsilon (x) g(x+\varepsilon y)\, dx=
 \int\limits _{Q-\varepsilon y} f(x) (g-m) (x+\varepsilon y)\,dx.
\]
Combining this with \eqref{3.14} and using the H\"{o}lder inequality to obtain 
\[
 \left|\int\limits_Q f_\varepsilon g\,dx\right|\leqslant
 \int\limits_{\N[R]^d} \varphi (y) \|f-m\|_{L_q(Q-\varepsilon y)}
 \|g\|_{L_{q'}(\N[R]^d)}\,dy.
\]

Taking here infimum over all polynomials $m$ and then supremum over all 
$g\in\P_{k-1}^\perp(Q)$ we get by \eqref{3.13}
\[
 E_k(f_\varepsilon;Q;L_q)\leqslant
 \int\limits_{\N[R]^d} E_k(f;Q-\varepsilon y;L_q)\varphi (y)\, dy.
\]

The result is done.
\end{proof}

Now we prove \eqref{3.7}. To this aim we first estimate 
$\var{}_p^k(f;\Delta;L_q)$, see \eqref{3.4}, for a disjoint family $\Delta$ of 
cubes $ \widehat Q\subset  Q$.
Due to \eqref{3.11} the Minkowski inequality gives for such $\Delta$
\[
 \var{}_p^k(f_\varepsilon;\Delta;L_q)\leqslant
 \int\limits_{\N[R]^d} \var{}_p^k (f;\Delta-\varepsilon y;L_q)\varphi(y)\,dy.
\]

Since $\Delta-\varepsilon y:=\{\widehat Q-\varepsilon y;\widehat Q\in\Delta\}$ 
is a disjoint 
family of cubes containing for small $\varepsilon$ in $Q^d$, the 
right-hand side is bounded by $\var{}_p^k(f;L_q)$. Taking then supremum over 
all such $\Delta$, we obtain the inequality
\[
 \var{}_p^k(f_\varepsilon; Q;L_q)\leqslant \var{}_p^k(f;L_q)
\]
that clearly implies \eqref{3.7}. $\qedsymbol$
\begin{Remark}\label{Rema-3.5}\em
Let $Q^d$ be an extension domain of $V_{pq}^k$, i.e., there exists a 
linear continuous extension operator from $V_{pq}^k$ int $V_{pq}^k(\widetilde 
Q)$ where $Q^d\subset \widetilde Q$ and $dist (Q^d;\N[R]\setminus 
\widetilde Q)>0$. In this case, $f_\varepsilon$ can be defined on the unit cube 
$Q^d$ and therefore Theorem~\ref{Teo-3.3} holds for $Q^d$ substituted for 
$Q$. 
\end{Remark}

Unfortunately, the corresponding extension theorem is unknown though it is true 
for some spaces $V_{pq}^k$, e.g., for $s(V_{pq}^k)=k$. The special case of the 
last assertion for the space $BV(Q^d)$ and even for more general class of 
domains is presented, e.g., in \cite[Th,~5.6]{Z}.

This remark leads to the following:

\subSect{} \textit{Conjecture}. For every $f\in V_{pq}^k$ there is a sequence 
$\{f_j\}\subset C^\infty(\N[R]^d)$ such that
\[
  \|f-f_j\|_{L_q}\to 0
  \text{ as } j\to\infty
\]
and, moreover,
\begin{equation}\label{3.15}
 \lim\limits_{i\to\infty} \var{}_p^k (f_j;L_q)\leqslant
 \var{}_p^k(f;L_q)
\end{equation}

\subSect{} \textit{Polynomial Approximation on $d$-Rings}. 
\begin{Teoo}\label{Teo-3.7}
Let $Q\subset Q^*$ be dyadic subcubes of $Q^d$. Then it is true that
\begin{equation}\label{3.16}
  E_k(f;Q^*\setminus Q;L_q)\leqslant
 c(k,d) \var{}_p^k (f;Q^*\setminus Q;L_q).
\end{equation}
\end{Teoo}
\begin{proof}
We need several auxiliary results.
\begin{lemma}\label{Lem-3.8}
Let $S_1,S_2\subset \N[R]^d$ be subsets of finite measure such that for 
$\varepsilon>0$
\begin{equation}\label{3.17}
 |S_1\bigcap S_2|\geqslant \varepsilon\cdot 
 \min\limits_{i=0,1} \{|S_i|\}.
\end{equation}
Then the following is true:
\begin{equation}\label{3.18}
 E_k(f;S_1\bigcup S_2;L_q)\leqslant
 c\varepsilon^{-k+1}
 \sum\limits_{i=1}^2 E_k(f;S_i;L_q).
\end{equation}
\end{lemma}

For proof see, e.g., \cite[p.~123]{BB}.
\end{proof}
\begin{lemma}\label{Lem-3.9}
Let $\{S_j\}_{1\leqslant j\leqslant N} $ be a family of subsets in $\N[R]^d$ of 
finite measure such that for some $\varepsilon>0$
\begin{equation}\label{3.19}
 |S_j\bigcap S_{j+1}|\geqslant
 \varepsilon\min\{|S_j|,|S_{j+1}|\}, \,\,
 1\leqslant j<N.
\end{equation}

Then it is true that
\begin{equation}\label{3.20}
  E_k(f;\bigcup\limits_{j=1}^NS_j;L_q) \leqslant
 c\sum\limits_{j=1}^N E_k(f,S_j;L_q)
\end{equation}
where $c=(c(k,d)\varepsilon^{-k+1})^{N-1}$.
\end{lemma}
\begin{proof}[Proof \em (induction on $N$)] For $N=2$ the result is done by 
\eqref{3.18}. Now assume that \eqref{3.20} holds for $N\geqslant2$ and prove 
it 
for $N+1$.

Setting $S^M:=\bigcup\limits_{j=1}^MS_j$ we get from \eqref{3.18}
\begin{multline*}
 |S^N\bigcap S_{N+1}|\geqslant \varepsilon |S_N\bigcap S_{N+1}|
 \geqslant \displaybreak[2]\\
 \geqslant\varepsilon
 \min\{|S_N|,|S_{N+1}|\}=
 \varepsilon \min\{|S^N|,|S_{N+1}|\}.
\end{multline*}

Further, Lemma~\ref{Lem-3.8} implies
\[
 E_k(f;\bigcup\limits_{j=1}^{N+1}S_j;L_q)\leqslant
 c(k,d) \varepsilon^{-k+1}(E_k(f;S^N;L_q)+
 E_k(f;S_{N+1};L_q))
\]
while the induction hypothesis gives 
\[
 E_k(f;S^N;L_q)\leqslant
 (c(k,d)\varepsilon ^{-k+1})^{N-1}
 \sum\limits_{j=1}^N E_k(f;S_j;L_q).
\]
Combining we get the result for $N+1$.
\end{proof}

Now we formulate the mentioned covering result proved in Appendix~I.
\begin{lemma}\label{lem-3.10}
There exists a covering $\K$ of $Q^*\setminus Q$ by cubes such that the 
following holds:

For every overlapping pair\footnote{i.e., intersecting by a set of positive 
measure}
$K_1,K_2\in\K$ 
\begin{equation}\label{3.21}
 |K_1\bigcap K_2|\geqslant
\frac12 \min \{|K_1|,|K_2|\},
\end{equation}
and, moreover,
\begin{equation}\label{3.22}
 \card \K\leqslant 4(2^d-1).
\end{equation}
\end{lemma}

Now we complete the proof of Theorem~\ref{Teo-3.7}.

By Lemmas \ref{Lem-3.9} and \ref{lem-3.10} we have
\[
 E_k(f;Q^*\setminus Q;L_q)\leqslant
 \left(c(k,d)2^{k-1}\right)^{4(2^d-1)}
 \sum\limits_{K\in\K}
 E_k(f;K;L_q).
\]

Moreover, by the definition of $(k,p)$-variation, see (\ref{3.1}$^a$),
\[
 E_k(f;K;L_q)\leqslant
 \var{}_p^k (f;Q^*\setminus Q;L_q)
\]
for every $K\in\K$.

Together with the previous inequality this gets the required result
\[
 E_k(f;Q^*\setminus Q;L_q)\leqslant
 c(k,d) \var{}_p^k (f;Q^*\setminus Q;L_q).
 \qed
\]

\section{Proof of Theorem~\ref{Teo-2.1}}

\subSect{} We begin with part (a) of this result and then derive from (a) part 
(b).

Let $f\in V_{pq}^k(:=V_{pq}^k(Q^d))$ where 
\begin{equation}\label{4.1}
 1\leqslant p<q<\infty,\,\,
 d\geqslant2 
 \text{ and }
 0<s:=s(V_{pq}^k)\leqslant k.
\end{equation}

Without loss of generality we assume that 
\begin{equation}\label{4.2}
 |f|_{V_{pq}^k} =1.
\end{equation}

Under these assumptions we given $N\in\N$ prove existence of a covering 
$\Delta_N$ of $Q^d$ by at most $N$ dyadic cubes and a piecewise polynomial 
$g_N\in \P_{k-1}(\Delta_N)$ such that 
\begin{equation}\label{4.3}
 \|f-g_N\|_q\leqslant
 c(k,d)N^{-s/d}.
\end{equation}

We first prove the result for $C^\infty$ functions and then derive from there 
\eqref{4.3}. For this aim we use the algorithm whose 
detailed exposition we present in Appendix~II.

An important ingredient of the algorithm is a \textit{weight} $W$ defined of 
the $\sigma$-algebra $A(\P[D])$ generated by dyadic cubes of $Q^d$. This by 
definition is a function $W:A(\P[D])\to\N[R]_+$ satisfying the conditions.

(\textit{Subadditivity}) For a disjoint family $\{S_i\}\subset A(\P[D])$
\begin{equation}\label{4.4}
 \sum W(S_i)\leqslant W(\bigcup S_i).
\end{equation}

(\textit{Absolutely continuity})
\begin{equation}\label{4.5}
 \lim\limits_{|S|\to0} W(S)=0.
\end{equation}

We \textit{normalize} $W$ by
\begin{equation}\label{4.6}
  W(Q^q)=1.
\end{equation}
To prove
Theorem~\ref{Teo-2.1}(a) for $f\in V_{pq}^k \bigcap C^\infty$ we 
define a weight $W$ by
\begin{equation}\label{4.7}
 W(S):=\var{}_p^k(f;S;L_q)^p,
 \quad
 S\in A(\P[D]).
\end{equation}

Due to Propositions~\ref{Pro-3.1}, \ref{Pro-3.2} and \eqref{4.2} $W$ 
satisfies the required properties \eqref{4.4}--\eqref{4.6}.

\subSect{} \textit{Description of the algorithm}. In the construction of the 
algorithm, we essentially exploit the canonical \textit{graph structure} of the 
set $\P[D]$ regarding as a graph with $\P[D]$ as the 
\textit{vertex set} and the \textit{edge set} consisting of pairs 
$\{Q',Q\}\subset\P[D]$ such that $Q'\subset Q$ and $|Q'|=2^{-d}|Q|$. In this 
case, we use the notation $Q'\to Q$ and call $Q'$ a \textit{son} of $Q$ and $Q$ 
the \textit{father} of $Q'$.

The set of all $2^d$ sons of $Q$ is denoted by $\P[D]_1(Q)$. This clearly is 
the 
uniform partition of $Q$ into $2^d$ congruent subcubes.

Further, a \textit{path} in the graph $\P[D]$ is a sequence
\begin{equation}\label{4.8}
  P:=\{Q_1\to Q_2\to\ldots Q_n\}.
\end{equation}
The vertices (cubes) $Q_1$, $Q_n$ are called the \textit{tail} and the 
\textit{head} of 
$P$, respectively. Moreover, we use the notations
\begin{equation}\label{4.9}
 P:=[Q_1,Q_n],\,\,
 Q_1=:T_P=:P^-,\,\,
 Q_n=:H_P=:P^+.
\end{equation}

It is readily seen that the following is true.
\begin{Pro}\label{Pro-4.1}
If $Q'\subset Q$ are dyadic cubes of $\P[D]$, there exists a \textit{unique} 
path joining $Q'$ and $Q$.
\end{Pro}

In terms of Graph Theory, $\P[D]$ is a \textit{rooted tree} with the 
\textit{root} $Q^d$.

More generally, the set $\P[D](Q)$ of all dyadic subcubes of $Q\in\P[D]$ is a 
\textit{rooted tree} with the \textit{root} $Q$.

Given $N\in\N$ and $W$, see \eqref{4.7}, the subset of ``bad'' cubes of $\P[D]$ 
is given by 
\begin{equation}\label{4.10}
  G_N:=\{Q\in\P[D];\,\,
 W(Q)\geqslant N^{-1}\};
\end{equation}
clearly, $Q^d\in G_N$, see \eqref{4.6}, and $G_N$ is finite, see \eqref{4.5}.

The algorithm gives the following partition of $G_N$ into the set of (basic) 
paths, see Proposition~\ref{Pro-II.2} of Appendix~II for the proof.
\begin{Pro}\label{Pro-4.3}
There exists a partition $\B$ of the set $G_N\setminus \{Q^d\}$ into paths such 
that\footnote{among basic paths may be singletons. In this case, $H_B=T_B$ and 
$W(H_B\setminus T_B)=0$.}
\begin{equation}\label{4.11}
 W(H_B\setminus T_B)< N^{-1},
 \quad
 B\in\B_N,
\end{equation}
and, moreover,
\begin{equation}\label{4.12}
 \card \B_N\leqslant 3N+1.
\end{equation}
\end{Pro}

Now we decompose the remaining part of $\P[D]$
\begin{equation}\label{4.13}
  G_N^c:=\P[D]\setminus G_N.
\end{equation}

To this end we define the \textit{boundary} of $G_N$ denote by $\partial G_N$ 
that consists of all \textit{maximal cubes} of $G_N^c$ with respect to the 
set--inclusion order.

In other words, every $Q'\in\P[D]$ containing $Q\in \partial G_N(\subset 
G_N^c)$ as a proper subset belongs to $G_N$. In particular, if $Q^+$ is the 
father of $Q\in \partial G_N$, then
\begin{equation}\label{4.14}
 W(Q)<N^{-1}
 \text{ and }
 W(Q^+)\geqslant N^{-1}.
\end{equation}
\begin{Pro}\label{Pro-4.4}
(a) The family $\{\P[D](Q); Q\in \partial G_N\}$ is disjoint and 
\begin{equation}\label{4.15}
 G_N^c=\bigcup\limits_{Q\in\partial G_N} \P[D](Q),
\end{equation}
i.e., the family is a partition of $G_N^c$.

(b) The following is true
\begin{equation}\label{4.16}
  \card (\partial G_N)\leqslant 2^d N.
\end{equation}
\end{Pro}
\begin{proof}
(a) Maximal cubes are either disjoint or coincide. Hence, $\partial G_N$ is a 
disjoint family.

Further, the cubes $Q\in \partial G_N$ are the roots of the trees $\P[D](Q)$ 
from \eqref{4.15}. Since the roots are disjoint, the corresponding trees also 
are, i.e., $\{\P[D](Q);Q\in \partial G_N\}$ is a disjoint family.

To prove that the family is a partition of $G_N^c$ we check that every $Q'\in 
G_N^c$ belongs to some $\P[D](Q)$ where $Q\in \partial G_N$.

Let $Q'=:Q_1\to Q_2\to\ldots\to Q^d$ be the path joining $Q'$ and $Q^d$, and 
$Q_i$, $i\geqslant2$, is the smallest cube of the path belonging to $G_N$. Then 
its son $Q_{i-1}$ belongs to $G_N^c$, i.e., $Q_{i-1}$ is maximal, and 
$Q'\in\P[D](Q_{i-1})$ as required.

(b) Let $Q^+$ be the father of $Q\in\partial G_N$ and $(\partial 
G_N)^+:=\{Q^+;Q\in \partial G_N\}$. Since $Q^+$ is unique, the set $(\partial 
G_N)^+$ is disjoint.

Further, every father has $2^d$ sons and therefore
\begin{equation}\label{4.17}
 \card(\partial G_N)\leqslant 2^d\card (\partial G_N)^+.
\end{equation}
Finally, \eqref{4.14}, subadditivity of $W$ and \eqref{4.6} imply
\[
 N^{-1}\card(\partial G_N)^+<
 \sum\limits_{Q\in (\partial G_N)^+} W(Q)\leqslant W(Q^d)=1.
\]

This and \eqref{4.17} give \eqref{4.16}.
\end{proof}

\subSect{} \textit{Definition on $N$ Term Approximation for $C^\infty$ 
Functions.} Theorem \ref{Teo-2.1}(a) will be derived from the next key result.
\begin{Pro}\label{Pro-4.5}
Let $f\in V_{pq}^k\bigcap C^\infty(\N[R]^d)$ where $d,p,q, s=s(V_{pq}^k)$ 
satisfy \eqref{4.1} and \eqref{4.2}. Given $N\in\N$ there exist a covering 
$\Delta_N\subset \P[D]$ of $Q^d$ and a piecewise polynomial 
$g_N\in\P_{k-1}(\Delta_N)$ such that
\begin{equation}\label{4.18}
 \|f-g_N\|_q\leqslant c(k,d) N^{-1/d}
\end{equation}
and, moreover,
\begin{equation}\label{4.19}
 \card \Delta_N\leqslant c(d)N.
\end{equation}
\end{Pro}
\begin{proof}
To introduce the family $\Delta_N$ we use the algorithm for the weight $W$ 
given by \eqref{4.7}.

Since $W$ satisfies the assumptions of Proposition~\ref{Pro-4.3}, see 
\eqref{4.4}--\eqref{4.6}, it determines the \textit{finite} set $G_N\subset 
\P[D]$, and the algorithm gives the partition $\B_N$ of $G_N\setminus \{Q^d\}$ 
into the basic paths.

Now the required covering $\Delta_N$ is given by
\begin{equation}\label{4.20}
 \Delta_N:=\{Q^d\}\bigcup\left( \bigcup\limits_{B\in\B_N}
 \{T_B,H_B\}\bigcup \P[D]_1(T_B)\right).
\end{equation}

Due to \eqref{4.12}
\begin{equation}\label{4.21}
 \card \Delta_N\leqslant 1+2(3N+1)+2^d(3N+1)
\end{equation}
which gives
\begin{equation}\label{4.22}
 \card \Delta_N\leqslant c(d)N,
 \quad N\in\N,
\end{equation}
where $c(d):=2^{d+2}+9$.

Now we define the required $g_N\in\P_{k-1}(\Delta_N)$ using for this aim 
polynomials of best approximation
determined by
\begin{equation}\label{4.23}
\|f-m_S\|_{L_q(S)}=E_k(f;S;L_q). 
\end{equation}
Further, we use for brevity the following notations
\begin{equation} \label{4.24}
 M_Q:=\sum\limits_{Q'\in \P[D]_1(Q)} m_{Q'}\cdot1_{Q'}
 -
 m_Q\cdot 1_Q,
\quad
 Q\in\P[D];
\end{equation}
and, moreover,
\begin{equation}\label{4.25}
  B^+:=H_B,
 \quad
 B^0:=H_B\setminus T_B,
 \quad 
 B^-:=T_B.
\end{equation}

Using this we write
\begin{equation}
\label{4.26}
 g_N:=m_{Q^d}+
 \sum\limits_{B\in\B_N} [(m_{B^+}-m_{B^0})\cdot1_{B^+}+
 (m_{B^0}-m_{B^-})\cdot 1_{B^-}+M_{B^-}]. 
\end{equation}

This clearly is a piecewise polynomial of degree $k-1$ over $\Delta_N$, see 
\eqref{4.20}.

Let us note that for $V$ being a singleton $B^{\pm}=\{B\}$, $B^0=\emptyset$, 
i.e., the corresponding terms in \eqref{4.26} and \eqref{4.20} equal 
$M_{B}$ and $\left\{\{B\},\P[D]_1(\{B\})\right\}$, 
respectively.\renewcommand{\qed}{}
\end{proof}

To estimate $f-g_N$ we need a suitable presentation of this difference; the 
next lemmas are used for its derivation.
\begin{lemma}\label{Lem-4.6}
Let $f\in L_q(Q)\bigcap C(\N[R]^d)$, $1\leqslant q\leqslant \infty$, $Q\in 
\P[D]$. Then the following holds
\begin{equation}\label{4.27}
 f=m_Q+\sum\limits_{Q'\in\P[D](Q)} M_{Q'}
\end{equation}
with convergence in $L_q(Q)$. 
\end{lemma}
\begin{proof}
Let $\P[D]_j(Q)$, $j\in\N[Z]$, be the partition of $Q$ into $2^{jd}$ congruent 
(dyadic) cubes, e.g., $\P[D]_0(Q)=\{Q\}$ and $\P[D]_1(Q)$ is the set of sons 
for $Q$. Then $P_j\in\P_{k-1}(\P[D](Q))$ is defined by 
\begin{equation}\label{4.28}
 P_j:=\sum\limits_{Q'\in\P[D]_j(Q)}
 m_{Q'}\cdot 1_{Q'}.
\end{equation}
We show that 
\begin{equation}\label{4.29}
  f-m_Q=
 \sum\limits_{j\geqslant0} (P_{j+1}-P_j)
 \text{ (convergence in $L_q(Q)$)}.
\end{equation}
Let $s_n$ be the $n$-th partial sum of the series \eqref{4.28}. Then
\[
 f-m_Q-s_n=f-P_{n-1}=
 \sum\limits_{Q'\in\P[D]_n(Q)} (f-m_{Q'})\cdot 1_{Q'}.
\]
This and \eqref{4.23} imply that 
\[
 \|f-m_Q-s_n\|_q=
 \left\{\sum\limits_{Q'\in\P[D]_n(Q)}
 \|f-m_{Q'}\|_{L_q(Q')}^q 
 \right\}^{1/q}=
 \left\{
 \sum\limits_{Q'\in\P[D]_n(Q)} E_k(f;Q';L_q)\right\}^{1/q}.
\]

By Thm.~4 of \cite[\S~2]{B-71} the right--hand side is bounded by 
$c(k,d)\omega_k\left( f;\frac{|Q|^{1/d}}{2^n}; L_q(Q)\right)$. 
Since this bound tends to $0$ as $n\to \infty$ for 
$q<\infty$ and for $q=\infty$ and $f\in C(\N[R]^d)$, \eqref{4.29} is done.
\end{proof}

Rewriting the second sum here by \eqref{4.27} we have
\[
 f-m_{Q^d}=
 \sum\limits_{B\in\B_N}
 \sum\limits_{Q\in B} M_Q+
 \sum\limits_{Q\in \partial G_N} (f-m_Q)\cdot 1_Q.
\]
Extracting from here the equality \eqref{4.26} for $g_N$ we obtain the required 
presentation
\begin{equation}\label{4.30}
 f-g_N=
 \sum\limits_{B\in\B_N} S_B+
 \sum\limits_{Q\in \partial G_N} (f-m_Q)\cdot1_Q;
\end{equation}
here we set\footnote{$S_B=0$ if $B$ is a singleton}
\begin{equation}\label{4.31}
  S_B=\left(\sum\limits_{Q\in B\setminus
 \{B^-\}} M_Q \right)- 
 [(m_{B^+}-m_{B^0})\cdot 1_{B^+}+
 (m_{B^0}-m_{B^-})\cdot1_{B^-}].
\end{equation}

The next result describes the basic properties of $\{S_B;B\in\B_N\}$.
\begin{lemma}\label{lem-4.7}
The following is true
\begin{equation}\label{4.32}
 S_B= \sum\limits_{Q\in B\setminus \{B^-\}} 
 \sum\limits_{Q'\in \P[D]_1(Q)\setminus B}
 (m_{B^0} - m_{Q'})\cdot 1_{Q'}.
\end{equation}
\end{lemma}
\begin{proof}
We begin with the identity
\begin{multline}\label{4.33}
 \sum\limits_{Q\in B\setminus\{B^-\}} M_Q=
\\
=\sum\limits_{Q\in B\setminus\{B^-\}}
 \sum\limits_{Q'\in\P[D]_1(Q)\setminus B}
 [(m_{Q'}- m_{B^+})\cdot 1_{Q'}+
 (m_{B^-}-m_{B^+})\cdot 1_{B^-}]
\end{multline}
proving by induction on $\card B$.

Let $B:=[Q_1,Q_n]=\{ Q_1\to Q_2\to\ldots Q_n\}$, i.e., $B^-:=Q_1$, $B^+:=Q_n$. 
Since $\P[D]_1(Q)\setminus B$ for $Q\in B\setminus\{B^-\}$ consists of all sons 
of 
$Q$ excluding the son belonging to $B$, 
\[
 \P[D]_1(Q_i)\setminus B=
 \P[D]_1(Q_i)\setminus\{Q_{i-1}\},\,\,
 i\geqslant2.
\]
Denoting the right--hand side by $\P[D]^*(Q_i)$ we then rewrite \eqref{4.33} as 
follows.
\begin{equation}\label{4.34}
  \sum\limits_{i=2}^n M_{Q_i}
  =\sum\limits_{i=2}^n
 \sum\limits_{Q\in \P[D]_1^*(Q_i)}
 [(m_Q-m_{Q_n})\cdot 1_Q]+
 (m_{Q_1}-m_{Q_n})\cdot 1_{Q_1}.
\end{equation}
For $n=2$ the right--hand side of \eqref{4.34} equals
\begin{multline*}
 \sum\limits_{Q\in \P[D]^*_1(Q_2)} 
 [(m_Q-m_{Q_2})\cdot 1_Q]+
 (m_{Q_1}-m_{Q_2})\cdot 1_{Q_1}:=\\
 \sum\limits_{Q\in \P[D]_1(Q_2)} m_Q\cdot 1_Q-
 m_{Q_2}\left(\sum\limits_{Q\in D_1^*(Q_2)}1_Q+1_{Q_1}\right).
\end{multline*}
Since $\P[D]_1^*(Q_2)$ is a partition of $Q_2\setminus Q_1$, the sum in the 
brackets equals $1_{Q_2\setminus Q_1}+1_{Q_1}=1_{Q_2}$. 
Hence, the right--hand side here equals $M_{Q_2}$, see \eqref{4.24}, as 
required.

Now let \eqref{4.33} hold for all paths of cardinality $n\geqslant2$. To prove 
it for $n+1$ we write \eqref{4.34} for the $n$-term path $\{Q_2\to\ldots 
Q_{n+1}\}$ and add to it \eqref{4.34} for $n=2$ written equivalently as follows.
\begin{equation*}
  m_{Q_2}=
  \sum\limits_{Q\in \P[D]_1^*(Q_2)}
  (m_Q-m_{Q_{n+1}})\cdot 1_Q+
  (m_{Q_{n+1}}-m_{Q_2})\cdot 1_{Q_2}+
  (m_{Q_1}-m_{Q_{n+1}})\cdot 1_{Q_1}.
\end{equation*}
Together with the equality
\begin{equation}\label{4.36}
  \sum\limits_{i=3}^{n+1} M_{Q_i}=
  \sum\limits_{i=3}^{n+1}
  \sum\limits_{Q\in \P[D]_1^*(Q_i)} (m_Q-m_{Q_{n+1}})\cdot 1_Q
  +(m_{Q_2}-m_{Q_{n+1}})\cdot 1_{Q_2}
\end{equation}
this gives
\begin{equation}\label{4.37}
  \sum\limits_{i=2}^{n+1} M_{Q_i}=
  \sum\limits_{i=2}^{n+1} 
  \sum\limits_{Q\in \P[D]_1^*(Q_i)}
  (m_Q-m_{Q_{n+1}})\cdot 1_Q+R
\end{equation}
where we set
\[
  R:=(m_{Q_{n+1}}- m_{Q_2})\cdot 1_{Q_2}+
  (m_{Q_1}-m_{Q_{n+1}})\cdot 1_{Q_1}
  +
  (m_{Q_2}-m_{Q_{n+1}})\cdot 1_{Q_2}=
\]
\[
 =
  (m_{Q_1}-m_{Q_{n+1}})\cdot 1_{Q_1}.
\]
Hence, \eqref{4.37} proves the required equality \eqref{4.33} for $n+1$.

Now we transform \eqref{4.34} by adding and extracting 
$m_{B^0}(:=m_{Q_n\setminus Q_1})$. This gives
\[
 \sum\limits_{i=2}^n M_{Q_i}=
 \sum\limits_{i=2} ^n \sum\limits_{Q\in \P[D]_1^*(Q_i)}
 (m_Q-m_{B^0})\cdot 1_Q+
 (m_{B^0}-m_{Q_n}) 
 \sum\limits_{i=2} ^n 
 \sum\limits_{Q\in \P[D]_1^*(Q_i)} 1_Q+
\]
\[
 +
 (m_{Q_1}-m_{B^0})\cdot 1_{Q_1}+
 (m_{B^0}-m_{Q_n})\cdot 1_{Q_1}.
\]
Since the second sum here equals $\displaystyle 
\sum\limits_{i=2}^n 
1_{Q_i\setminus Q_{i-1}}= 1_{Q_n\setminus Q_1}$ and, in the chosen notations, 
see \eqref{4.30}, 
\begin{equation}\label{4.36-a}
 S_B:=\sum\limits_{i=2}^n M_{Q_i}-
 (m_{Q_n}-m_{B^0})\cdot 1_{Q_n}-
 (m_{B^0}-m_{Q_1})\cdot 1_{Q_1},
\end{equation}
these two equalities give 
\[
  S_B=\sum\limits_{i=2} ^n
  \left[ \sum\limits_{Q\in \P[D]_1^*(Q_i)} (m_Q-m_{B^0}) \cdot 1_Q\right]+R
\]
where the remainder $R$ equals
\begin{multline}\label{4.37-b}
  R:= [(m_{B^0}-m_{Q_n})\cdot 1_{Q_n\setminus Q_1} +
  (m_{B^0}-m_{Q_n})\cdot 1_{Q_1}\\
 +
  (m_{Q_1}-m_{B^0})\cdot 1_{Q_1}]-[(m_{Q_n}-m_{B^0})\cdot 1_{Q_n}+
 (m_{B^0}-m_{Q_1})\cdot 1_{Q_1}].
\end{multline}
Since the brackets here annihilate, $R=0$.  

The identity \eqref{4.33} is done.\renewcommand{\qedsymbol}{}
\end{proof}

\subSect{} \textit{Proof Proposition~\ref{Pro-4.5}}. We should prove
that for $f\in V_{pq}^k\bigcap C^\infty (\N[R]^d)$
\begin{equation}\label{4.38}
  \|f-g_N\|_q\leqslant
  c(k,d) N^{-s/d}.
\end{equation}

Due to the presentation \eqref{4.28}
\begin{equation}\label{4.39}
  \|f-g_N\|_q\leqslant
  \left\| \sum\limits_{B\in\B_N} S_B\right\|_q+
  \left\|\sum\limits_{Q\in \partial G_N} (f-m_Q)\cdot 1_Q\right\|_q
\end{equation}
and remains to estimate each sum.

To complete the proof of Lemma~\ref{lem-4.7} it remains to replace 
$\displaystyle \sum\limits_{Q\in B\setminus \{B^-\}} M_Q$ in \eqref{4.31} by 
the right--hand side of the identity \eqref{4.33}. Since the terms outside the 
double sum annihilate we obtain the required presentation of $S_B$, see 
\eqref{4.32}. \qed

\begin{lemma}\label{Lem-4.8}
(a) Supports of the functions $S_B$, $B\in\B_N$, are disjoint.

(b) It is true that
\begin{equation}\label{4.40}
  \|S_B\|_q\leqslant
  e(k,d)\var{}_p^k(f;B^+\setminus B^-; L_q).
\end{equation}
\end{lemma}
\begin{proof}
(a) Since $supp\,S_B=B^+\setminus B^-$, see Lemma~\ref{Lem-4.8}, the 
supports of 
$S_B$ and $S_{\widetilde B}$ no intersect if their heads do not. Otherwise, one 
of these (dyadic) cubes, say, $\widetilde B^+$, embeds into the other. Then 
$\widetilde B^+$ embeds into the tail $B^-$ of the path $B$. Hence, 
$supp\,S_{\widetilde B}$ does not intersect $supp\,S_B=B^+\setminus B^-$.

(b) By the identity \eqref{4.32}
\[
 S_B=
 \sum\limits_{Q\in B^*}
 \sum\limits_{Q'\in \P[D]_1^*(Q)}
 (m_{B^0}-f+f-m_{Q'})\cdot 1_{Q'}
\]
where we for brevity set $B^*:=B\setminus\{B^-\}$.

Further, we have
\[
 S_B=(m_{B_0}-f) \sum\limits_{Q\in B^*}
 \sum\limits_{Q'\in \P[D]_1^*(Q)}1_Q+
 \sum\limits_{Q\in B^*} 
 \sum\limits_{Q'\in \P[D]_1^*(Q)} (f-m_{Q'})\cdot 1_{Q'}.
\]
Since the family $\displaystyle \bigcup\limits_{Q\in B^*} \P[D]_1^*(Q)$ is a 
partition of $B^+\setminus B^-$, the sum of indicators here equals 
$1_{B^+\setminus B^-}$ and the equality implies
\[
 \|S_B\|_q\leqslant
 \|f-m_{B^0}\|_{L_q(B^+\setminus B^-)}+
 \left(\sum\limits_{Q\in B^*}
 \sum\limits_{Q'\in \P[D]_1^*(Q)}
 \|f-m_{Q'}\|^q\right)^{1/q}=
\]
\[
 =E_k(f;B^+\setminus B^-;L_q)+
 \left(\sum\limits_{Q\in B^*}
 \sum\limits_{Q'\in \P[D]_1^*(Q)}
 E_k(f;Q';L_q)^q\right)^{1/q}.
\]

By the Jenssen inequality the second term is bounded by 
\[
 \left\{
 \sum\limits_{Q\in B^*}
 \sum\limits_{Q'\in \P[D]_1^*(Q)}
 E_k(f;Q';L_q)^p\right\}^{1/p}.
\]
Since the family $\displaystyle \bigcup\limits_{Q\in B^*} \P[D]_1^*(Q)$ is a 
partition of $B^+\setminus B^-$, this sum is bounded by 
$\var{}_p^k(f;B^+\setminus B^-;L_q)$, see the definition of $(k,p)$--variation 
in (\ref{3.1}$^a$).

Moreover, by Theorem~\ref{Teo-3.7}
\[
 \|f-m_{B^0}\|_{L_q(B^0)}:=
 E_k(f;B^+\setminus B^-)\leqslant c(k,d)\cdot 
 \var{}_p^k(f;B^+\setminus B^-;L_q).
\]

Combining this with the previous inequality we obtain \eqref{4.40}. 
\qed

Now we use Lemma~\ref{Lem-4.8} to estimate the first term in \eqref{4.39}. We 
have 
\[
 \left\|\sum\limits_{B\in\B_N} S_B\right\|_q
 \leqslant\left\{\sum\limits_{B\in\B_N}\|S_B\|_q^q\right\}^{1/q}
 \leqslant
\]
\[
 \leqslant
 c(k,d)\left\{\sum\limits_{B\in\B_N} \var{}_p^k (f;B^+\setminus B^-; L_q)^q
 \right\}^{1/q}.
\]

Moreover, by the definition of the weight $W$, see \eqref{4.7}, and the 
inequality \eqref{4.11} of Proposition~\ref{Pro-4.3}
\[
 \var{}_p^k(f;B^+\setminus B^-;L_q):=
 W(B^+\setminus B^-)^{1/p}\leqslant N^{-1/p}. 
\]
Combining with the previous inequality and using \eqref{4.12} we finally have 
the required estimate
\[
 \left\|\sum\limits_{B\in\B_N} S_B\right\|_q
 \leqslant c(k,d) \left(N^{-q/p}
 \card \B_N\right)^{1/q}
 \leqslant
 c(k,d) \left(N^{-q/p}(3N+1)\right)^{1/q}
 \leqslant
\]
\[
 \leqslant
 c_1(k,d)N^{-1/p+1/q}:=c_1(k,d)N^{-s/d}.
\]

It remains to obtain the similar bound for the sum over boundary $\partial G_N$ 
in \eqref{4.39}.

Due to Proposition~\ref{Pro-4.4} and \eqref{4.14} $\partial G_N$ is disjoint 
and, moreover,
\[
  \var{}_p^k(f;Q;L_q)^p=: W(Q)<N^{-1}
\]
for every $Q\in\partial G_N$.

This immediately implies
\[
 \left\|\sum\limits_{Q\in \partial G_N}
 (f-m_Q)\cdot 1_Q\right\|_q=
 \left\{ \sum\|f-m_Q\|_{L_q(Q)}^q
 \right\}^{1/q}:=
 \left\{\sum\limits_{Q\in \partial G_N}
 E_k(f;Q;L_q)^q\right\}^{1/q}\leqslant
\]
\[
 \leqslant\left\{ \sum\limits_{Q\in \partial G_N}
 \var{}_p^k (f;Q;L_q)^{q/p}\right\}^{1/p}
 \leqslant N^{-1/p} \left(\card G_N\right)^{1/q}.
\]

Since $\card G_N\leqslant 2^d N$, see \eqref{4.16}, this finally gives
\[
 \left\|\sum\limits_{Q\in \partial G_N} 
 (f-m_Q)\cdot 1_Q \right\|_q\leqslant
 2^{d/q} N^{-s/d}
\]
as required.

Proposition~\ref{Pro-4.5} is done.
\end{proof}

\subSect{} \textit{Proof of Theorem~\ref{Teo-2.1}(a)}. We derive the result 
from Theorem~\ref{Teo-3.3} and Proposition~\ref{Pro-4.5}.

Let $Q:=[1-\delta,\delta)$, $\delta>0$, and $f\in V_{pq}^k$ if $q<\infty$ and 
$f\in V_{pq}^k \bigcap C(\N[R]^d)$ if $q=\infty$.

Given $\varepsilon>0$ Theorem~\ref{Teo-3.3} then yields a functions 
$f_\varepsilon\in C^\infty(\N[R])$ such that 
\begin{equation}\label{4.42}
  \|f-f_\varepsilon\|_{L_q(Q)}\leqslant\varepsilon
\end{equation}
and, moreover,
\begin{equation}\label{4.43}
 \var{}_p^k(f_\varepsilon;Q;L_q)\leqslant|f|_{V_{pq}^k}.
\end{equation}

Since Proposition~\ref{Pro-4.5} is homothety--invariant, it remains true for 
$Q$ substituted for $Q^d$. Hence, given $N\in\N$ there exist a covering 
$\widetilde \Delta_N\subset \P[D](Q)$ of $Q$ and a piecewise polynomial 
$\widetilde 
g_N\in\P_{k-1}(\widetilde\Delta_N)$ such that 
\begin{equation}\label{4.44}
  \|f_\varepsilon-\widetilde g_N\|_{L_q(Q)}
  \leqslant c(k,d) N^{-s/d}
  \var{}_p^k(f;Q;L_q)
\end{equation}
and, moreover,
\begin{equation}\label{4.45}
  \card \widetilde\Delta_N\leqslant c(d)N.
\end{equation}

Now let $h$ be a homothety mapping $Q$ onto $Q^d$, i.e.,
\[
 h(x):=\frac{x-\delta e}{1-2\delta},
 \quad
 x\in \N[R]^d,
\]
where $e:=(1,\ldots,1)$.

Then $\Delta_N:=h(\widetilde\Delta_N)\subset \P[D](:=\P[D](Q^d))$ is a family 
covering of $Q^d$ and satisfying
\begin{equation}\label{4.46}
  \card \Delta_N=\card \widetilde \Delta_N\leqslant c(d)N;
\end{equation}
moreover, $g_N:=\widetilde g_N\circ h^{-1}$ is a piecewise polynomial from 
$\P_{k-1}(\Delta_N)$.

We will show that for $f\in V_{pq}^k$ with $q<\infty$ and for $f\in 
V_{p\infty}^k \bigcap C(\N[R]^d)$
\begin{equation}\label{4.47}
  \|f-g_N\|_q\leqslant c(k,d) N^{-s/d} |f|_{V_{pq}^k};
\end{equation}
this clearly implies Theorem~\ref{Teoo-1.2}(a) for $N\geqslant c(d)$, see 
\eqref{4.46}.

Let $h^*g:=g\circ h^{-1}$, $g\in L_q(Q^d)$. Then $h^*: L_q(Q^d)\to L_q(Q)$ and 
$\|h^*\|=(1-2\delta)^{d/q}$.

Further, we write
\[
 \|f-g_N\|_q\leqslant
 \|(f\circ h-f_\varepsilon)\circ h^{-1}\|_q+
 \|(f_\varepsilon-\widetilde g_N)\circ h^{-1}\|_q\leqslant
\]
\[
 \leqslant(1-2\delta)^{d/q}
 \left(\|f-f_\varepsilon\|_{L_q(Q)}
 +\|f-f\circ h\|_{L_q(Q)}
 +\|f_\varepsilon-\widetilde g_N\|_{L_q(Q)}\right).
\]
By \eqref{4.44} and \eqref{4.43} the third term in the brackets is bounded by 
$c(k,d)N^{-s/d}|f|_{V_{pq}^k}$ while the first tends to $0$ as 
$\varepsilon\to0$, see \eqref{4.42}, and the second as $\delta\to0$ for 
$q<\infty$, and also for $q=\infty$, if $f$ is uniformly continuous on $Q$.

This proves \eqref{4.47} and Theorem~\ref{Teoo-1.2}(a) for $N\geqslant c(d)$.

To obtain the result for $1\leqslant N<c(d)$ we simply set $\Delta_N:=\{Q^d\}$ 
and $g_N:=m_{Q^d}$. Then 
\[
 \|f-g_N\|_q=E_k(f;Q^d;L_q)<
 c(d)^{s/d} N^{-s/d}|f|_{V_{pq}^k}
\]
and, moreover, $\card \Delta_N=1\leqslant N$.

This gives Theorem~\ref{Teo-2.1}(a) for all $N\geqslant1$.\hfill \qed

\subSect{} \textit{Proof of Theorem~\ref{Teo-2.1}(b)}. We establish the analog 
of Theorem~\ref{Teo-2.1}(a) with a partition of $Q^d$ into at most $c(d)N$ 
$d$-rings.

To this aim we write the piecewise polynomial $g_N$ of 
Theorem~\ref{Teo-2.1}(a), see \eqref{4.26}, in the form
\begin{equation}\label{4.48}
  g_N:=m_{Q^d}+\sum\limits_{Q\in\Delta_N} P_Q\cdot 1_Q
\end{equation}
where $P_Q\in\P_{k-1}$ and 
\[
 \Delta_N:=\bigcup\limits_{B\in\B_N}
 (\{ H_B,T_B\} \bigcup \P[D]_1(B^-)),
\]
see \eqref{4.20}.

We first assume that $\Delta_N$ \textit{covers} $Q^d$. If $\Delta_N$ is no
partition (otherwise, result is clear), it contains at least one \textit{tower}
\[
 T:=\{Q_1\subsetneqq\ldots \subsetneqq Q_n\}\subset \Delta_N.
\]
This means that for every $0\leqslant i\leqslant n$ there in no $Q\in \Delta_N$ 
such that $Q_i\subsetneqq Q\subsetneqq Q_{i+1}$; here $Q_0:=\emptyset$, 
$Q_{n+1}:=Q^d$; hence, the \textit{bottom} $Q_1\not=\emptyset$ and the 
\textit{top} $Q_n$ are, respectively, minimal and maximal cubes of $T$ closest 
to $Q^d$.

According to this definition $(\Delta_N\setminus T)\bigcup\{Q_n\}$ still covers 
$Q^d$. Moreover, the tops of different towers are no intersect.

These, in particular, imply that if $T_j$, $1\leqslant j\leqslant m$, are all 
towers of $\Delta_N$ and $Q(T_j)$ are their tops, then
\[
 \left( \Delta_N\setminus \bigcup\limits_{j=1}^m
 T_j\right)
 \bigcup
 \left(\bigcup\limits_{j=1}^m
 Q(T_j)\right)
\]
is a partition of $Q^d$.

Hence, it suffices to subdivide each $Q(T_j)$ into a set of $d$-rings 
of cardinality card $T_j$. We do this for $m=1$ and then repeat 
the procedure for the remaining towers.

Now let $T:=\{Q_1\subsetneqq\ldots\subsetneqq Q_n\}$ be the single tower of 
$\Delta_N$. Setting
\[
 R_i:=Q_i\setminus Q_{i-1},
 \quad 1\leqslant i\leqslant n,
\]
where $R_1=Q_1$ as $Q_0:=\emptyset$, we obtain the partition 
$\P[R]_n:=\{R_i\}_{1\leqslant i\leqslant n}$ of $Q_n:=Q(T)$ into $d$-rings.

Further, we define the family of polynomials $\{P_{R_i}\}\subset \P_{k-1}$ 
given by $\displaystyle P_{R_i}:=\left(\sum\limits_{j=i}^n P_{Q_j}\right) 
\cdot1_{R_j}$, $1\leqslant i\leqslant n$. These definitions imply the identity
\begin{equation}\label{4.49}
  \sum\limits_{i=1}^n P_{Q_i}\cdot 1_{Q_i}=
  \sum\limits_{R\in\P[R]_n}P_R\cdot 1_R.
\end{equation}
Moreover, the $T$ is single in $\Delta_N$, hence, 
$\P[R]_n\bigcup(\Delta_N\setminus T)$ is a partition of $Q^d$ into $\leqslant 
n+(N-n)=N$ $d$-rings while the piecewise polynomial 
\[
 \widetilde g_N:=
 m_{Q^d}+
 \sum\limits_{R\in \P[R]_n} P_R\cdot 1_R+
\sum\limits_{Q\in \Delta_N\setminus T} P_Q\cdot 1_Q
\]
belongs to $\P_{k-1}(\P[R]_n\bigcup(\Delta_N\setminus T))$ and equals $g_N$ by 
\eqref{4.49} and \eqref{4.48}.

This gives the result for $\Delta_N$ being a covering of $Q^d$.

Now let $\Delta_N$ do not a covering of $Q^d$. Then $\P[D]_1(Q^d)\bigcap 
G_N^c\not=\emptyset$, since otherwise $\P[D]_1(Q^d)\subset G_N$, i.e., every 
son of $Q^d$ is the head of a basic path. By the definition of $\Delta_N$, see 
\eqref{4.20}, this implies that $\P[D]_1(Q)\subset \Delta_N$, i.e., $\Delta_N$ 
is a 
covering of $Q^d$, a contradiction.

Further, the set of heads $\P[D]_1(Q^d)\bigcap G_N$ contains in 
$\Delta_N\subset 
\P[D]\setminus\{Q^d\}$, and, moreover, nonempty as otherwise $\Delta_N=\{Q^d\}$.

Hence, the set
\[
 \widetilde\Delta_N:=\Delta_N\bigcup(\P[D]_1(Q^d)\bigcap G_N^c)
\]
is a covering of $Q^d$ and its cardinality is bounded by
\[
 N+\card \P[D]_1(Q^d)-1=N+2^d-1\leqslant 2^dN.
\]

To complete the proof it suffices to modify the $g_N$ to obtain $\widetilde 
g_N\in \P_{k-1}(\Delta_N\bigcup (\P[D]_1(Q^d)\bigcap G_N^c))$ such that 
\begin{equation}\label{4.50}
  \|f-\widetilde g_N\|_q\leqslant
  c(k,d) N^{-s/d} |f|_{V_{pq}^k}.
\end{equation}

We define $\widetilde g_N$ by 
\[
 \widetilde g_N:= g_N+
 \sum\limits_{Q\in \P[D]_1(Q^d)\bigcap G_N^c}
 (m_Q-m_{Q^d})\cdot 1_Q
\]
and then prove \eqref{4.50}. 

Substituting here the right--hand side of \eqref{4.48} for $g_N$ and using the 
notations
\[
  S:=\bigcup\limits_{Q\in \Delta_N} Q,
  \quad
  \widetilde\Delta:=\P[D]_1(Q^d)\bigcap G_N^c
\]
we have 
\[
 \widetilde g_N=
 g_N\cdot 1_S+
 \left( \sum\limits_{Q\in \widetilde \Delta}
  m_Q\cdot 1_Q\right)\cdot 1_{Q^d\setminus S}.
\]
This, in turn, implies
\begin{equation}\label{4.51}
  \|f-\widetilde g_N\|_q\leqslant
  \|f-g_N\|_q+
  \left\|\sum\limits_{Q\in \widetilde \Delta}
  (f-m_Q)\cdot 1_Q\right\|_q.
\end{equation}
The first summand is clearly bounded by the right--hand side of \eqref{4.50}.

Moreover, the second equals
\[
 \left(\sum\limits_{Q\in\widetilde\Delta }
 E_k(f;Q;L_q)^q\right)^{1/p}\leqslant
 \left\{
 \sum\limits_{Q\in \widetilde \Delta} E_k(f;Q;L_q)^p
 \right\}^{1/p}
 \leqslant
\]
\[
 \leqslant
 \left\{\sum\limits_{Q\in \widetilde \Delta}
 \var{}_p^k (f;Q;L_q)^p\right\}^{1/p}=:
 \left\{\sum\limits_{Q\in \widetilde \Delta}
 W(Q)\right\}^{1/p}.
\]
Since $\widetilde \Delta\subset G_N^c$, every $W(Q)< N^{-1}$.

Hence, the second summand in \eqref{4.51} is bounded by $(2^dN^{-1})^{1/p}$ 
that clearly also majorates by the right--hand side of \eqref{4.50}.

This proves \eqref{4.50} and Theorem~\ref{Teo-2.1}(b).
\hfill \qed

\begin{Remark}\label{Rema-4.9}\em
Let $\Sigma_{k,N}$ be a (nonconvex) set of piecewise polynomials given by
\begin{equation}\label{4.52}
  \Sigma_{k,N}:=
  \bigcup\limits_{\Delta\in\P[D]\setminus\{Q^d\} }
  \P_{k-1}(\Delta)
\end{equation}
where $\Delta$  runs over all coverings of $Q^d$ of cardinality at most $N$.

\newcommand{\R}{\P[R]}

Replacing here $\Delta$'s by partitions of $Q^d$ into at most $N$ 
$d$-rings to 
define the similar space $\R_{k,N}$.

Further, by $\sigma_{k,N}(f;L_q)$ and $\rho_{k,N}(f;L_q)$ we denote best 
approximation of $f$ in $L_q$ by $\Sigma_{k,N}$ and $\R_{k,N}$, respectively.

Then the first part of the proof for Theorem~\ref{Teo-2.1}(b) (with 
$\Delta_N$ being a covering) gives the following inequality
\[
 \rho_{k,N} (f;L_q)
 \leqslant
 \sigma_{k,N}(f;L_q).
\]
\end{Remark}

Since the used in the proof procedure is reversible, the converse inequality 
is also true. Hence,
\begin{equation}\label{4.53}
  \sigma_{k,N} (f;L_q)=\rho_{k,N}(f;L_q),\quad
  N\geqslant 2^d,
\end{equation}
as any covering $\Delta$ in \eqref{4.52} contains at 
least $2^d(=\card \P[D]_1(Q^d))$ cubes.

\section{Proofs of Corollaries}

\subSect{} \textit{Proof of Theorem~\ref{Teoo-1.2}}. We obtain this result from 
Theorem~\ref{Teo-2.1} with $s(V_{pq}^k)=k$ and $q<\infty$. The latter asserts 
that under the assumptions
\begin{equation}\label{5.1}
  d\geqslant2,
  \quad
  1\leqslant p<q<\infty
  \text{ and }
  \frac kd=\frac1p-\frac1q
\end{equation}
there exist a covering $\Delta_N\subset \P[D]$ of $Q^d$ by at most $N$ cubes 
and 
$g_N\in\P_{k-1}(\Delta_N)$ such that
\begin{equation}\label{5.2}
 \|f-g_N\|_q\leqslant
 c(k,d)N^{-k/d}|f|_{V_{pq}^k}.
\end{equation}
It remains to replace here $|f|_{V_{pq}^k}$ by the Sobolev seminorm 
$|f|_{W_p^k(Q^d)}$ if $p>1$ and by the $BV^k(Q^d)$ seminorm if $p=1$. 
This substitution is justified by the two--sided inequality
\begin{equation}\label{5.3}
 |f|_{V_{pq}^k}\approx
 \left\{
 \begin{array}{lcl}
  |f|_{W_p^k} & \text{if} & p>1,\\
  |f|_{BV^k}  & \text{if} & p=1, 
 \end{array}
 \right.
\end{equation}
where constants are independent of $f$, see Theorems~4 and 12 from 
\cite[\S~4]{B-71}.

The result is done.\, \qed

\subSect{} \textit{Proof of Theorem~\ref{Teo-2.3}}. We should prove the analog 
of the previous result for the homogeneous Besov $\dot B_p^\lambda(Q)$, 
$\lambda>0$, $Q\subset \N[R]^d$, whose associated seminorm is given by 
\begin{equation}\label{5.4}
  |f|_{B_p^\lambda(Q)}:=
  \left\{\int\limits_0^{|Q|^{1/d}}
  \left( \frac{\omega_k(f;t;L_p(Q))}{t^\lambda}\right)^\lambda
  \frac{dt}{t}\right\}^{1/p}
\end{equation}
where $k=k(\lambda):=\min\{n\in\N;n>\lambda\}$.

We derive this from Theorem~\ref{Teo-2.1} with $s(V_{pq}^k)=\lambda$, 
$k=k(\lambda)$ and $q<\infty$. Hence, in this case,
\begin{equation}\label{5.5}
  1\leqslant p<q<\infty,
  \quad
  \frac\lambda d=\frac1p-\frac1q,
\end{equation}
and Theorem~\ref{Teo-2.1} gives under these assumptions the inequality
\begin{equation}\label{5.6}
  \|f-g_N\|_q\leqslant
  c(k,d)N^{-\lambda/d} |f|_{V_{pq}^k}
\end{equation}
with the corresponding $g_N\in \P_{k(\lambda)-1}(\Delta_N)$ and $\Delta_N$.

It remains to replace here $|f|_{V_{pq}^k}$ by $|f|_{B_p^\lambda (Q^d)}$.

To this end we use the classical embedding theorem that under the assumptions 
\eqref{5.5} gives the inequality 
\begin{equation}\label{5.7}
  E_k(f;Q;L_q)\leqslant
  c(d,\lambda, q)|f|_{B_p^\lambda(Q)},
\end{equation}
see Remark~\ref{Remark-5.1} below for details.

Now let $\Delta: =\{Q\}$ be a disjoint family of cubes from $Q^d$. Then 
\eqref{5.7} implies
\[
  \left(\sum\limits_{Q\in\Delta}
  E_k(f;Q;L_q)^p\right)^{1/p}\leqslant
  c(d,\lambda,q)
 \left(\sum\limits_{Q\in\Delta}
  (|f|_{B_p^\lambda(Q)})^p\right)^{1/p}.  
\]
Due to Lemma 2 from \cite[\S~5]{B-94} the sum in the right--hand side is 
bounded by $c(k,d)|f|_{B_p^\lambda(Q^d))}$. Taking supremum over $\Delta$ we 
then obtain the required inequality
\begin{equation}\label{5.8}
  |f|_{V_{pq}^k} \leqslant
  c(k,\lambda,q)|f|_{B_p^\lambda(Q^d)}
\end{equation}
and prove Theorem~\ref{Teo-2.3}. \hfill \qed

\subSect{} \textit{Proof of Theorem~\ref{Teo-2.4}}. We now deal with the 
homogeneous space $\dot B_p^{\lambda1}(Q)$ which associated seminorm is given by
\begin{equation}\label{5.9}
  |f|_{B_p^{\lambda1}(Q)}:=
  \int\limits_0^{|Q|^{1/d}}
  \frac{\omega_k(f;t;L_p(Q))}{t^{\lambda+1}}dt
\end{equation}
where $k=k(\lambda)$.

We prove, under the conditions
\begin{equation}\label{5.10}
  1\leqslant p<q=\infty,
  \quad d\geqslant2\,
  \text{ and }\,\frac\lambda d=\frac1p,
\end{equation}
existence of the corresponding $\Delta_N$ and $g_N\in \P_k(\Delta_N)$ such that 
the next inequality is true:
\begin{equation}\label{5.11}
  \|f-g_N\|_\infty\leqslant
  c(d,\lambda,p)
  N^{-\lambda/d} |f|_{B_p^{\lambda1}(Q^d)};
\end{equation}
here $k=k(\lambda)$.

Due to \eqref{5.10} $\lambda=\frac dp\leqslant d$ and therefore  
$k(\lambda)\leqslant d+1$. Since the norms $\|f\|_{B_p^{\lambda1}(Q)}:= 
\|f\|_{L_p(Q)}+ |f|_{B_p^{\lambda1}(Q)}$ with different $k\geqslant k(\lambda)$ 
are equivalent it suffices to prove \eqref{5.11} for $k:=d+1$ instead of 
$k(\lambda)$.

We derive \eqref{5.11} from Theorem~\ref{Teo-2.1}(a) with $s(V_{pq}^k)=\lambda$ 
and $q=\infty$. This requires the embedding
\begin{equation}\label{5.12}
 \dot B_p^{\lambda1 }(Q^d)\subset
 V_{p\infty}^k(Q^d)\bigcap C(\N[R]^d),
\end{equation}
since Theorem~\ref{Teo-2.1} with $q=\infty$ holds only for $f\in V_{p\infty}^k 
\bigcap C(\N[R]^d)$. But $C(\N[R]^d)$ in \eqref{5.12} can be discarded, for the 
condition \eqref{5.10} implies the embedding $B_p^{\lambda1}(Q^d)\subset 
C(\N[R]^d)|_{Q^d}$, see, e.g., \cite[Thm~6.8.9(a)]{BL}.

By a reason explained later we begin with the case
\begin{equation}\label{5.13}
  \dot B_p^{\lambda1}(\N[R]^d)\subset V_{p\infty}^k(\N[R]^d),
  \quad k=d+1.
\end{equation}
This will be proved for $p=1$ and $\infty$ while the general case will be then 
derived from these by the real interpolation.

If $p=1$, then \eqref{5.10} implies $\lambda=d$ and $k(\lambda)=d+1$; moreover, 
by definition $\dot B_1^{\lambda1}= \dot B_1^\lambda$. In this case \eqref{5.7} 
is still true, i.e., we have
\begin{equation}\label{5.14}
  E_{d+1}(f;Q;L_\infty)\leqslant
  c(k,d)|f|_{B_1^d(Q)},
\end{equation}
see Remark~\ref{Remark-5.1} below.

Using the argument used in the proof of \eqref{5.8} we obtain from \eqref{5.14} 
the required inequality
\begin{equation}\label{5.15}
  |f|_{V_{1\infty}^{d+1}(\N[R]^d)}
  \leqslant
  c(d) |f|_{B_1^d(\N[R]^d)}.
\end{equation}
This proves \eqref{5.13} for $p=1.$

Now let $p=\infty$, hence, $\lambda=\frac dp=0$. Appearing space $\dot B_\infty 
^{01}(\N[R]^d)$ is defined by the seminorm
\[
  |f|_{B_\infty^{01}(\N[R]^d)}:=
  \sum\limits_{j\in\N[Z]}\|f*\varphi_j\|_{L_\infty(\N[R]^d)}
\]
where $\{\varphi_j\}$ is a sequence of test functions, satisfying, in 
particular, the condition
\[
  f=\sum\limits_j f*\varphi_j
\]
with convergence in the distributional sense, see, e.g., \cite[sec.~6.3]{BL}.

This implies
\[
 |f|_{V_{\infty\infty}^{d+1}(\N[R]^d)}:=
  \sup\limits_{Q\subset \N[R]^d} E_d(f;Q;L_\infty)
 \leqslant
 \|f\|_{L_\infty(\N[R]^d)}
 \leqslant
 \sum\limits_j \|f*\varphi_j\|_{L_\infty(\N[R]^d)}
 = |f|_{B_\infty^{01}(\N[R]^d)}.
\]

Hence, we prove \eqref{5.13} for $p=\infty$ as well.

Interpolating the embeddings obtained we then have
\begin{equation}\label{5.16}
  (\dot B_\infty^{01}, \dot B_1^{d1})_{\theta p}
  \subset (V_{\infty\infty}^{d+1},
  V_{1\infty}^{d+1})_{\theta p};
\end{equation}
hereafter $\N[R]^d$ is skipped for brevity.

Taking $\theta:=\frac\lambda d$ and $p=d$ we obtain for the left--hand side the 
embedding
\begin{equation}\label{5.17}
  \dot B_p^{\lambda 1}\subset (\dot B_\infty^{01}, \dot B_1^{d1})_{\theta p}
\end{equation}
see \cite[Ch~5, Thm~6(9)]{P}.

Now we show that the right--hand side contains in $V_{p\infty}^{d+1}(\N[R]^d)$ 
with 
$p:=\frac d\lambda$.

Let $\P[E]: L_\infty(\N[R]^d)\to l_\infty(\Delta)$ be a map given by 
\[
 \P[E]:f\mapsto(E_d(f;Q;L_\infty))_{Q\in\Delta};
\]
here $\Delta$ is a disjoint family of cubes $Q\subset \N[R]^d$.

By definition
\[
 \|\P[E](f)\|_{l_p(\Delta)}:=
 \left( \sum\limits_{Q\in\Delta} E_d(f;Q;L_\infty)^p\right)^{1/p}
  \leqslant
  |f|_{V_{p\infty}^{d+1}(\N[R]^d)},
\]
i.e., $\P[E]$ maps $V_{p\infty}^{d+1}(\N[R]^d)$ into $l_p(\Delta)$ and 
$\|\P[E]\|\leqslant1$, $1\leqslant p\leqslant\infty$.

Interpolating this sublinear operator\footnote{i.e., $\P[E](f+g)\leqslant
\P[E](f)+\P[E](g)$ and $\P[E](\lambda 
f)=|\lambda|\P[E](f)$, $\lambda\in\N[R]$} by the real method we obtain
\[
 \|\P[E](f)\|_{(l_\infty(\Delta),l_1(\Delta))_{\theta p}}
 \leqslant
 |f|_{(V_{\infty\infty}^{d+1}, V_{p\infty}^{d+1})_{\theta p}},
\]
see, e.g., \cite[4.1.5(c)]{BK} for validity of the interpolation result for 
sublinear operators.

Moreover, $(l_\infty(\Delta), l_1(\Delta))_{\theta p}$ with $\theta=\frac1p$ 
equals $l_p(\Delta)$, see, e.g., \cite[Thm~5.6.1]{BL}.
Together with the previous this implies
\[
 \left(\sum\limits_{Q\in\Delta} E_d(f;Q;L_\infty)^p\right)^{1/p}
 \leqslant |f|_{(V_{\infty\infty}^{d+1}, V_{1\infty}^{d+1})_{\theta p}}
\]
where $\theta =\frac1p=\frac\lambda d$.

Taking here supremum over all $\Delta$ we obtain the embedding
\[
 (V_{\infty\infty}^{d+1}, V_{\infty1}^{d+1})_{\theta p}
 \subset V_{p\infty}^{d+1},
\]
implying the required embedding \eqref{5.13}.

To derive from \eqref{5.13} the similar embedding for $Q^d$ we use a bounded 
linear extension operator
\[
 Ext: \dot B_p^{\lambda 1}(Q^d)\to \dot B_p^{\lambda 1}(\N[R]^d),
\]
with $\|Ext\|\leqslant c(\lambda,d)$, see, e.g., \cite[Thm.~2.7.2]{BB}, and the 
restriction operator
\[
 Res: V_{p\infty}^{d+1}(\N[R]^d)\to V_{p\infty}^{d+1}(Q^d).
\]

Denoting the embedding operator in \eqref{5.13} by $U$ and composing it with 
the now introduced ones we obtain the operator $U_{Q^d}:=Ext\circ U\circ Res$ 
that 
embeds $\dot B_p^{\lambda 1}(Q^d)$ into $V_{p\infty} ^{d+1}(Q^d)$ with the 
embedding constant $\|Ext\|\leqslant c(d,\lambda)$.

This proves the required inequality \eqref{5.12} and, therefore, 
\mbox{Theorem~\ref{Teo-2.4}.\qed}
\begin{Remark}\label{Remark-5.1}\em
We prove the inequalities \eqref{5.7} and \eqref{5.14}.
\end{Remark}

Let $f\in L_p(Q)$, $1\leqslant p<q\leqslant \infty$, and $m_Q\in\P_{k-1}$ be 
best approximation of $f$ in $L_p(Q)$. Setting for brevity
\[
 \omega(t):=\omega_k(f;t;L_p(Q)),\quad t>0,
\]
we estimate the nonincreasing rearrangement of $f-m_Q$ as follows
\begin{equation}\label{5.18}
  (f-m_Q)^*(t)\leqslant c(k,d)
  \int\limits_{t/2}^{|Q|}
  \frac{\omega(u^{1/d})}{u^{1+1/p}}du,
  \quad
  t\leqslant|Q|,
\end{equation}
see \cite[Appendix~II, Cor.~2']{B-94}.

Taking $L_q$--norm and applying the Hardy inequality we have
\begin{equation}\label{5.19}
\begin{array}{l}
  \|f-m_Q\|_{L_q(Q)} =\|(f-m_Q)^*\|_{L_q(0,|Q|)}
  \leqslant \\
\leqslant \displaystyle
c(k,d)\|\Hh_{1/q}\|\left(
 \int\limits_0^{|Q|}
\left(\frac{\omega(u^{1/d})}{u^{1/p-1/q}}\right)^qdu\right)^{
1/q }
\end{array}
\end{equation}
where $\Hh_\mu$, $\mu>0$, is the Hardy operator given by
\[
 \Hh_\mu g(t):=t^\mu
 \int\limits_t^{|Q|} \frac{g(u)}{u^\mu}\frac{du}u.
\]
Since  $\|\Hh_\mu\|<\infty$ for $\mu>0$, the inequality \eqref{5.19} is 
true for $1/q>0$, i.e., for $q<\infty$.

Since $\frac1p-\frac1q=\frac\lambda d$, the integral in \eqref{5.19} is bounded 
by 
\begin{multline*}
 d^{1/q}\left(\int\limits_{0}^{|Q|^{1/d}} \left(\frac{\omega(t)}{t^\lambda}
 \right)^q\frac{dt}t \right)^{1/q}
 \leqslant
  c(k,\lambda) d^{1/q}
 \left(\int\limits_0^{|Q|^{1/d}} \left(\frac{\omega(t)}t\right)^p
 \frac{dt}t \right)^{1/p}=\\
=
 c(k,d,\lambda) |f|_{B_p^\lambda(Q)}.
\end{multline*}
Hence, for $q<\infty$
\[
 \|f-m_Q\|_{L_q(Q)}\leqslant c(k,d,\lambda)|f|_{B_p^\lambda(Q)}
\]
which implies \eqref{5.7} as the left--hand side is clearly bigger than
$E_k(f;Q;L_q)$.

For $q=\infty$ we pass in \eqref{5.18} to limit as $t\to+0$ to obtain
\[
 \|f-m_Q\|_{L_\infty(Q)} =\lim\limits_{t\to0} (f-m_Q)^*(t)\leqslant
 c(k,d)\int\limits_0^{|Q|} \frac{\omega(u^{1/d})}{u^{1/p}}\frac{du}u=
 d\cdot c(k,d)|f|_{B_p^{\lambda 1}(Q)}.
\]
Hence, \eqref{5.14} follows. \qed

\setcounter{section}{0}
\renewcommand{\thesection}{Appendix~\Roman{section}}
\renewcommand{\theequation}{\Roman{section}.\arabic{equation}}
\renewcommand{\theRemark}{\Roman{section}.\arabic{Remark}}
\renewcommand{\thePro}{\Roman{section}.\arabic{Remark}}

\section{Covering Dyadic Rings}

Let first $Q$ be a dyadic subcube of $Q^*$  such that 
\begin{equation}\label{I.1}
  dist(Q,\N[R]^d\setminus Q^*)>0.
\end{equation}

In this case Lemma~\ref{lem-3.10} asserts the following:
\begin{Teoo}\label{Teo-I.1}
There exists a covering $\K$ of $Q^*\setminus Q$ by cubes\footnote{ let us 
recall, that all cubes have a form $\displaystyle 
\prod\limits_{i=1}^d[a_i,b_i)$} such that for every overlapping\footnote{i.e., 
$|K_1\bigcap K_2|>0$.} pair $\{K_1,K_2\}\subset \K$ 
\begin{equation}\label{I.2}
  |K_1\bigcap K_2|\geqslant\frac12\min\limits_{i=1,2}|K_i|
\end{equation}
and, moreover,
\begin{equation}\label{I.3}
 \card\K=4(2^d-1).
\end{equation}
\end{Teoo}
\begin{proof}
Without loss of generality we assume that $Q^*=Q^d:=[0,1)^d$. By \eqref{I.1} 
the dyadic cube $Q$ contains in one of sons of $Q^d$, say, in 
$[1/2e,e):=\prod\limits_{i=1}^d[1/2,1)$, $e:=(1,1,\ldots)$. Denoting 
$Q:=\prod\limits_{i=1}^d[a_i,b_i)$ we, in particular, have
\begin{equation}\label{I.4}
  0<1-a_i\leqslant1/2,
  \quad 1\leqslant i\leqslant d.
\end{equation}

Now let $\pi$ denote a partition of $Q^d$ by hyperplanes passing through the 
vertex $a\in Q$ and parallel to the coordinate hyperplanes. It consists of 
$2^d$ parallelotops $\pi$  every of which consists a single vertex 
$\varepsilon\in\{0,1\}^d$ of $Q^d$. We numerates elements of $\pi$ by these 
vertices, so that $\pi:=\{\Pi_\varepsilon\}_{\varepsilon\in\{0,1\}^d}$ and 
$\varepsilon$ contains in the closure of $Q^d\bigcap \Pi_\varepsilon$. Then 
$\Pi_\varepsilon$ and $\Pi_{\varepsilon'}$ have a (unique) common face whenever 
$\varepsilon$, $\varepsilon'$ differ by a single coordinate. Moreover, in this 
case the edge $[\varepsilon,\varepsilon')$ of $Q^d$ is orthogonal to this face 
and intersects $\Pi_\varepsilon$ and $\Pi_{\varepsilon'}$.

Let $G(\pi)$ denote a graph with the vertex set $\pi=\{\Pi_\varepsilon\}$ and 
the edges consisting of pairs $\{\Pi_\varepsilon,\Pi_{\varepsilon'}\}$ with 
common face. The bijection 
$\varphi:\Pi_\varepsilon\longleftrightarrow\varepsilon$ is an isomorphism of 
$G(\pi)$ onto the \textit{hypercube graph} $\Gamma_d$ whose vertices and edges 
are those of the cube $Q^d$.

In fact, $\varepsilon=\varphi(\Pi_\varepsilon)$ and 
$\varepsilon'=\varphi(\Pi_{\varepsilon'})$ are joined by an edge in $\Gamma_d$ 
whenever $\varepsilon$ differs from $\varepsilon'$ by a single coordinate, 
i.e., whenever $\Pi_\varepsilon$ and $\Pi_{\varepsilon'}$ have a common face 
and therefore are joined by an edge in $G(\pi)$.

Further, the graph $\Gamma_d$ has a \textit{Hamiltonian cycle}, i.e., a cycle 
that visits each vertex of $\Gamma_d$ exactly once, see, e.g., \cite{HHW}. 
Therefore, $G(\pi)$ also has such a cycle denoted by $\C(\pi)$.

Now we apply  this construction to the parallelotop 
$\Pi_e:=\prod\limits_{i=1}^d [a_i,1)$ containing 
$Q=[a,b):=\prod\limits_{i=1}^d[a_i,b_i)$ and the vertex $b$ substituting for 
that of $a$. This give a partition $\widehat\pi$ of $\Pi_e$ into $2^d$ 
parallelotops 
one of which is $Q$. Then we numerate them by the vertex set $V$  of $\Pi_e$ 
such that $\widehat\pi=\{\Pi_v\}_{v\in V}$ and $v$ belong to the closure of 
$\Pi_v\bigcup \Pi_e$, e.g., $\Pi_a=Q$.

Using the partition $\widehat \pi$ we as above define the graph $G(\widehat 
\pi)$ isomorphic to $\Gamma_d$ and denote by $\C(\widehat\pi)$ the 
corresponding Hamiltonian cycle. Hence, $\Pi_v$, $\Pi_{v'}$ are neighbours in 
$\C(\widehat\pi)$ if they have a common face orthogonal to $[v,v']$.

Now we define a new graph $G$ with the vertex set
\[
 V(G):=\left(\pi\setminus \{\Pi_e\}\right) \bigcup
 \left(\widehat \pi\setminus\{\Pi_a\}\right)
\]
where $\Pi_a=Q$, and with edge set $E(G)$ of two parts. The first consists 
of 
edges from $G(\pi)$ and $G(\widehat\pi)$ such that both of their endpoints 
belong to either $\pi\setminus\{\Pi_e\}$ or $\widehat \pi\setminus\{\Pi_a\}$.

The second part is as follows. 

Let $\Pi_\varepsilon$, $\Pi_{\varepsilon'}$ from $\C(\pi)$ have common faces 
with $\Pi_e(\in\C(\pi))$. Since $\Pi_e:=[a,e)$, the vertex $a\in Q$ belongs to 
$\Pi_\varepsilon$ and to $\Pi_{\varepsilon'}$. Therefore there exist 
parallelotops $\Pi_v$ and $\Pi_{v'}$ from $\widehat\pi$ each having one of 
faces common with  that of $Q$ and another containing in $\Pi_\varepsilon$ and 
$\Pi_{\varepsilon'}$, respectively.

Then the pairs $\{\Pi_\varepsilon,\Pi_v\}$, $\{\Pi_{\varepsilon'},\Pi_{v'}\}$ 
from $V(G)$ form the remaining part of edges from $E(G)$.

It is now the matter of definition to check that 
\[
 \C=\C_1\bigcup \C_2:=
 (\C(\pi)\setminus\{\Pi_e\})
 \bigcup
 (\C(\widehat\pi)\setminus\{Q\})
\]
is a Hamiltonian cycle in $G$.

Now we construct the desired covering $\K$ of $Q^d\setminus Q$ beginning first 
with extension of each parallelotop of $\C_i$, $i=1,2$, to a cube containing in 
$Q^d\setminus Q$.

We begin with the set
\[
 \C_1:=\{\Pi_\varepsilon;\varepsilon\in\{0,1\}^d\setminus\{e\}\}
\]
containing $2^d-1$ elements.

Let $\Pi_\varepsilon:=\prod\limits_{i=1}^d [a_i^\varepsilon, b_i^\varepsilon)$ 
and $l^\varepsilon$ be the maximal edgelength of $\Pi_\varepsilon$. Since by 
definition of $\Pi_\varepsilon$ every edge $[a_i^\varepsilon, b_i^\varepsilon)$ 
equals either $A_i:=[0,a_i)$ or $B_i:=[a_i,1)$ and $|A_i|\geqslant |B_i|$, see 
\eqref{I.4}, the maximal edge of $\Pi_\varepsilon$, say, $[a_{i_0}^\varepsilon, 
b_{i_0}^\varepsilon)$, has the form
\begin{equation}\label{I.5}
  [a_{i_0}^\varepsilon, b_{i_0}^\varepsilon)=
  A_{i_0}=[0,a_{i_0}).
\end{equation}
Now we extend $\Pi_\varepsilon$ to a cube replacing every edge 
$[a_i^\varepsilon, b_i^\varepsilon)=A_i$ by $[\widehat a_i^\varepsilon, 
\widehat b_i^\varepsilon):=[0,a_{i_0})$ and every edge equal to $B_i$ by 
$[\widehat a_i^\varepsilon, \widehat b_i^\varepsilon):=[1-a_{i_0},1)$.

In this way, we obtain the cube 
\[
 Q_\varepsilon:=\prod\limits_{i=1}^d
 [\widehat a_i^\varepsilon, \widehat b_i^\varepsilon)\subset Q^d
\]
of edgelength $a_{i_0}$ that contains $\Pi_\varepsilon$ and, moreover, is 
contained in $Q^d\setminus \Pi_e$.

In fact, the projections of $Q_\varepsilon$ and $\Pi_e$ on the $x_{i_0}$ --
axis are $[\widehat a_i^\varepsilon, \widehat b_i^\varepsilon)=[0,a_{i_0})$, 
see \eqref{I.5}, and $[a_{i_0},1)$, respectively, that do not intersect.

Thus, we have 
\[
 \bigcup\C_1=\bigcup\limits_{\varepsilon\neq e}Q_\varepsilon,
 \quad
 Q_\varepsilon\supset \Pi_\varepsilon,
 \quad
 \varepsilon\neq e,
\]
where $\bigcup \C_1:=\bigcup\{\Pi;\Pi\in\C_1\}$.

Further, we cover $\bigcup \C_2$ similarly. By definition
\[
 \C_2=\left\{\Pi_v:=\prod\limits_{i=1}^d
 [a_i^v,b_i^v);
 \,\, v\in V\setminus\{a\}\right\}
\]
where $[a_i^v,b_i^v)$ equals either $A_i:=[b_i,1)$ or $B_i:=[a_i,b_i)$.

Let us show that $|A_i|\geqslant|B_i|$. In fact, $Q$ is a dyadic cube, say, 
$Q:=2^{-n}(\alpha+Q^d)$, $\alpha\in\N[Z]_+^d$, and therefore $|A_i|=2^{-n}$ 
while $|B_i|=1-b_i=2^{-n}(2^n-\alpha_i-1)\geqslant 2^{-n}$ as $b_i<1$. 

Then the maximal edge of $\Pi_v$, say, $[a_{i_0}^v,b_{i_0}^v)$ has the form
\begin{equation}\label{I.6}
  [a_{i_0}^v,b_{i_0}^v)=A_{i_0}=[b_{i_0},1).
\end{equation}

Now we extend $\Pi_v$ replacing every $[a_i^v,b_i^v)=A_i$ by $[\widehat a_i^v, 
\widehat b_i^v):=[1-l^v,1)$ and every $[a_i^v,b_i^v)=B_i$ by $[b_i,b_i-l^v)$; 
here 
$l^v=1-b_{i_0}$ is the maximal edge length of $\Pi_v$.

In this way, we obtain the cube
\[
 Q_v:=\prod\limits_{i=1}^d [\widehat a_i^v, \widehat b_i^v)
 \subset Q^d
\]
of volume $(l^v)^d$ that contains $\Pi_v$ and, moreover, is contained in 
$Q^d\setminus Q$. 

In fact, the embedding $\Pi_v\subset Q_v$ follows from the inequality 
$a_i\geqslant \widehat a_i^v:=b_i-l^v$ equivalent to 
\[
 |B_i|=b_i-a_i\leqslant |A_i|\leqslant l^v.
\]
Further, $Q_v\bigcap Q=\emptyset$, as the projections on the $x_{i_0}$ -- axis 
of these cubes $[a_{i_0}^v, b_{i_0}^v)=[b_{i_0},1)$ and $[a_{i_0},b_{i_0})$, 
respectively, do not intersect.

Thus, we have
\[
 \bigcup \C_2\subset 
 \prod\limits_{v\neq a} Q_v\subset Q\setminus Q^d
 \,\text{ and }\,
 \Pi_v\subset Q_v.
\]
This gives the family $\F:=\{Q_\varepsilon\}\bigcup\{Q_v\}$ of $2(2^d-1)$ cubes 
covering the $d$-ring $Q^d\setminus Q$ such that $Q_\varepsilon$, $Q_v$ are 
uniquely defined by the corresponding $\Pi_\varepsilon\supset Q_\varepsilon$, 
$\Pi_v\supset Q_v$ from the Hamiltonian cycle $\C$.

Further, we numerate the cycle $\C$ by integers to obtain
\[
 \C=\{\Pi_i;1\leqslant i\leqslant 2\cdot 2^d-1\}
\]
where $\Pi_i:=\Pi_1$ for $i=2\cdot 2^d-1$, such that $\Pi_i$, $\Pi_{i+1}$ are 
neighbours in $\C$. Hence, they adjoint to some edge of $Q^d$ denoted by 
$[v_i,v_{i+1})$ such that a small shift along this edge of the smaller 
parallelotop remains in $\Pi_i\bigcup\Pi_{i+1}\subset Q^d
\setminus Q$.

Let now $\{Q_i;1\leqslant i\leqslant 2\cdot 2^d-1\}$ where $Q_i:=Q_1$ for 
$i=2\cdot 2^d-1$ and the numeration of the family $\F$ is induced by of $\C$.

Then by definition of cubes from $\F$ the following is true.

(a) $\bigcup\limits_i Q_i$ covers $Q^d\setminus Q$;

(b) cubes $Q_i\supset \Pi_i$, $Q_{i+1}\supset \Pi_{i+1}$ adjoint to the edge 
$[v_i,v_{i+1})$ and the shift along this edge of the smaller one, say $Q_i$, by 
its length remains in $Q_{i+1}\subset Q^d\setminus Q$.

Let then $Q_{i+1/2}$ denote the image of $Q_i$ under such a shift by the 
one--half of its length. Then the covering $\K:=\{Q_i,Q_{i+1/2}\}$ of 
$Q^d\setminus Q$ consists of $4(2^d-1)$ cubes satisfying the inequality
\[
 |Q_j\bigcap Q_{i+1/2}|
 \geqslant 
 1/2 \min\{|Q_j|, |Q_{i+1/2}|\}
\]
for $j=i,i+1$.

The result is done, for $Q$ containing in the interior $Q^d$, see \eqref{I.1}. 
\end{proof}
\begin{Remark}\label{R-I.2}\em
Now let \eqref{I.1} do not hold. Then some face of $Q$ contains in a face of 
$Q^d$ with the vertex $e$. To reduce this case to the previous we introduce the 
pair $Q$, $Q^*:=[0,2)^d$. Clearly, $Q$ is a dyadic subcube of $Q^*$ and 
\eqref{I.1} holds for this pair.

Applying to $\{Q,Q^*\}$ the first step of the previous procedure we obtain the 
families of parallelotops $\{\Pi_\varepsilon\}_{\{0,2\}^d\setminus\{2e\}}$ and 
$\{\Pi_v\}_{V\setminus\{a\}}$ covering $Q^*\setminus\Pi_{2e}$ and 
$\Pi_{2e}\setminus Q$, respectively (here $V$ is the vertex set of $\Pi_{2e}$).

Then we use the second step of the procedure extending all of the parallelotops 
$\Pi_\varepsilon\bigcap Q^d$, $\varepsilon\neq 2e$, $\Pi_v\bigcap Q^d$, $v\neq 
a$, to the corresponding cubes. In some cases, the extension is fictional as 
the corresponding intersections are empty. E.g., if the vertex $b$ of $Q$ 
coincides with $e$, then $\Pi_v\bigcap Q^d=\emptyset$ if $v\neq a$.

At the third step we correspondingly shift the obtained cubes forming the 
pairs 
$\{Q_i, Q_{i+1/2}\}$, $1\leqslant i \leqslant 2(2^d-1)$; here $Q_i=Q_1$ for 
$i=2\cdot 2^d-1$ and the numeration of the family of cubes $\{Q_i\}$ is induced 
by the natural numeration of the Hamiltonian cycle $\C$ generated by the pair 
$\{Q,Q^*\}$. 

Discarding empty pairs $\{Q_i,Q_{i+1/2}\}$ we obtain the required covering $\K$ 
of $Q^d\setminus Q$ satisfying \eqref{I.2}. In this case, $2(2^d-1)\leqslant 
\card \K<4(2^d-1)$.

Hence, Lemma~\ref{lem-3.10} is true also in this case.
\end{Remark}

\section{Approximation Algorithm}

\quad
We describe the algorithm giving as output the covering $\Delta_N$ in 
Theorem~\ref{Teo-2.1}. In the forthcoming text, we freely use terms and 
definitions of section~4.2, e.g., weight, dyadic tree $\P[D]:=\P[D](Q^d)$, 
paths etc. 
Proofs of some formulated below statements will be remained to the reader (all 
of them presented in details in \cite[\S~6]{B-2004}).

Let $W:A(\P[D])\to\N[R]_+$ be a subadditive absolutely continuous weight normed 
by the condition
\begin{equation}\label{II.1}
 W(Q^d)=1.
\end{equation}

Then the set
\begin{equation}\label{II.2}
 G_N:=\{Q\in \P[D];
 W(Q)\geqslant N^{-1}\},\quad N\in\N,
\end{equation}
is a \textit{finite rooted subtree} of $\P[D]$ with the root $Q^d$. Hence, 
every 
path connecting $Q\in G_N$ and $Q^d$ is unique and belongs to $G_N$.

Further, let $G_N^{\min}$ be the set of minimal elements of $G_N$ with respect 
to the set--inclusion order.

Hence, every $Q\in G_N$ contains properly some minimal cube and a son $Q'$ of 
such a cube satisfies
\[
  W(Q')<N^{-1}.
\]
In particular, $G_N^{\min} $ is disjoint and as every disjoint subset of $G_N$ 
has at most $N$ elements.

Somehow numerating $G_N^{\min}$, say,
\[
 G_{\min}^N:=\{Q_i\}_{1\leqslant i\leqslant m_N}
\]
where
\begin{equation}\label{II.3}
  m_N:=\card G_N^{\min}\leqslant N,
\end{equation}
we then denote by $L_i$ a (unique) path in $G_N$ joining $Q_i$ and $Q^d$.

By the definition of $G_N^{\min}$ 
\begin{equation}\label{II.4}
  G_N=\bigcup\limits_{i=1}^{m_N} L_i.
\end{equation}

We divide each $L_i$ into more small paths
\[
  P_i:=L_i\setminus \bigcup\limits_{j=0}^{i-1}L_j,
 \quad
 1\leqslant i\leqslant m_N
\]
where $L_0:=\{Q^d\}$.
\begin{lemma}\label{Lem-II.1} \em(\cite[p.~164]{B-2004})\em{}
(a) The family $\{P_i\}_{1\leqslant i\leqslant m_N}$ is a partition of 
$G_N\setminus\{Q^d\}$.

(b) Every $P_i$ is of the form
\begin{equation}\label{II.5}
  P_i:=[Q_i,Q_i^c):=[Q_i,Q_i^c]\setminus\{Q_i^c\}
\end{equation}
where  $Q_i^c$ is the tail of a path $L_i\bigcap L_j$ with $j<i$.
\end{lemma} 

The set
\[
 \C_N:=\{Q^d\}\bigcup\{Q_i^c\}_{1\leqslant i\leqslant m_N}
\]
contains at most $m_N+1$ elements called \textit{contact cubes}.

Now we refine $G_N$ subdividing each $P_i$ by contact cubes from $P_i\bigcap 
\C_N$. In this way, we define a set of subpaths $[Q',Q'')$ where $Q'$ is either 
a minimal cube or a contact cube, and $Q''$ is a contact cube.

Denoting the set of these subpaths by $\P_N$ we obtain from \eqref{II.3} 
\begin{equation}\label{II.6}
 \card \P_N\leqslant 2m_N+(m_N+1)=3m_N+1.
\end{equation}
We finally divide each path $P\in\P_N$ to the required \textit{basic paths}. To 
this aim, we use an auxiliary weight defined on paths $P=[T_P,H_P]$ of $\P[D]$ 
by 
\begin{equation}\label{II.7}
  \widetilde W(P):=W(H_P\setminus T_P).
\end{equation}

Now we define for each $P\in\P_N$ a family of vertices (cubes) 
$\{Q_i(P)\in P;1\leqslant i\leqslant i_p\}$ using  induction on $i$.

We begin with $Q_1(P):=T_P$ and then having $Q_i(P)$ define $Q_{i+1}(P)$ as a 
vertex in the half--open from the left path 
\[
 (Q_i(P),H_P]:=
 [Q_i(P),H_P]\setminus\{Q_i(P)\}
\]
satisfying the conditions
\[
 \widetilde W([Q_i(P),Q_{i+1}(P)])\geqslant N^{-1},
\]
\[
 \widetilde W([Q_i(P), Q_{i+1}(P)))<N^{-1}.
\]
Then we define the $i$-th basic path $B_i(P)$ by setting
\begin{equation}\label{II.8}
  B_i(P):=[Q_i(P),Q_{i+1}(P)).
\end{equation}

The vertex $Q_{i+1}(P)$ may be undetermined, if 
\[
 \widetilde W([Q_i(P),H_p])< N^{-1}.
\]
In this case, we complete induction setting $i_P:=i$ and defining $B_i(P)$ to 
be equal $[Q_i(P),H_P]$. However, to preserve the formula \eqref{II.8} for this 
case, we define $Q_{i+1}(P)$ as the \textit{father} of $H_P$. Denoting it, say, 
$H_P^+(\in P)$ we define $B_i(P)$ for this case by \eqref{II.8} with 
$Q_{i+1}(P):= H_P^+$ and $i:=i_P$.

Hence, the induction has been completed with $Q_{i+1}(P)=H_P$ or 
$Q_{i+1}(P)=H_P^+$ for $i=i_P$.
In this way, we obtain a partition of $P$ by subpaths $B_i(P):= 
[Q_i(P),Q_{i+1}(P))$, $1\leqslant i\leqslant i_P$.
Let us emphasize that if $Q_{i+1}=H_p^+$, then $B_i(P)$ may be a singleton 
$\{H_p\}$.

By definition these subpaths satisfy
\begin{equation}\label{II.9}
\begin{array}{c}
 \widetilde W([Q_i(P),Q_{i+1}(P)))<N^{-1},\\
 \widetilde W([Q_i(P_i),Q_{i+1}(P)))\geqslant N^{-1}
\end{array}
\end{equation}
for $1\leqslant i\leqslant i_P-\varepsilon_P$ where $\varepsilon_P:=0$ if 
$Q_{i_P+1}(P)=H_P^+$ and $\varepsilon_P:=1$ otherwise; in the first case, only 
the first of inequalities \eqref{II.9} holds.

Collecting all the basic paths we obtain the refinement of $\P_N$ given by 
\begin{equation}\label{II.10}
 \B_N:=\{B_i(P);\,\,1\leqslant i\leqslant i_P,\,\,P\in\P_N\}.
\end{equation}

The next result ($=$Proposition~\ref{Pro-4.3}) gives the output of the 
algorithm.
\begin{Pro}\label{Pro-II.2}
(a) $\B_N$ is a partition of $G_N\setminus \{Q^d\}$.

(b) For every $B=[T_B,H_B]\in \B_N$
\begin{equation}\label{II.10a}
 W(H_B\setminus T_B):=\widetilde W(B)<N^{-1}.
\end{equation}
(c) The following is true 
\begin{equation}\label{II.11}
  \card \B_N\leqslant 3N+1.
\end{equation} 
\end{Pro}
\begin{proof}
(a) $\B_N$ is a refinement of the partition $\P_N$, hence, also a partition.

(b) is given by the first inequality in \eqref{II.9} and the definition of 
$B_i(P)$.

(c) Let $\{P_i\}$ be a strictly monotone sequence of subpaths in a path $P$, 
i.e.. the head of $P_i$ is a \textit{proper} subset of the tail of $P_{i+1}$. 
Then by definition of $\widehat W$, see \eqref{II.7},
\[
 \sum\limits_i\widetilde W(P_i)\leqslant W(H_P\setminus T_P).
\]
Now let $B_i(P):=[Q_i(P),Q_{i+1}(P))$, $1\leqslant i\leqslant i_P$, is the 
partition of $P\in \P_N$ into the basic paths. By the second inequality 
\eqref{II.9}
\begin{equation}\label{II.12}
  (i_P-\varepsilon_P)N^{-1}\leqslant
  \sum\limits_{i=1}^{i_P-\varepsilon_P}
  \widetilde W([Q_i,Q_{i+1}]).
\end{equation}
Since the sequence $\{[Q_i,Q_{i+1}]\}_{1\leqslant i\leqslant 
i_P-\varepsilon_P}$ has multiplicity 2, it can be divided into two strictly 
monotone subsequences. Hence, the right--hand side of \eqref{II.12} is bounded 
by $2W(H_P-T_P)$. This implies
\[
 \card \B_N=\sum\limits_{P\in\P_N} i_P\leqslant
 2N\sum\limits_{P\in \P_N} W(H_P)+
 \sum\limits_{P\in\P_N}\varepsilon_P.
\]
Since the set $\{H_P\}_{P\in \P_N}$ is disjoint, $\sum\limits_{P\in\P_N}W(H_P) 
\leqslant W(Q^d)=1$.

Further, $\varepsilon_P=1$ if and only if the endpoint of $B_i(P)$ with $i=i_P$ 
is $H_P^+$. By the definition of $\P_N$ every head $H_P$ of $P\in\P_N$ is a 
contact cube. Hence, 
\[
 \sum\limits_{P\in\P_N}\varepsilon_P\leqslant \card\C_N\leqslant
 m_N+1\leqslant N+1,
\]
see \eqref{II.3}.

Combining this with the previous estimates we finally get 
\[
 \card \B_N\leqslant2N+N+1=3N+1.\qed
\]\renewcommand{\qedsymbol}{}
\end{proof} 

\newcommand{\MYNE}[1]{\textit{#1}}

\end{document}